\documentclass[12pt]{amsart}
\usepackage{amssymb,latexsym,amsmath,tikz-cd}
\usepackage{enumitem}
\usepackage[mathscr]{eucal}
\usepackage{bbm,xcolor,comment}
\usepackage{mathdots}
\usepackage{chngcntr}

\usepackage{xr}
\numberwithin{equation}{section}
\numberwithin{table}{section} 

\newcommand{\hsp}[1]{{\hbox{\hspace{#1}}}}

\newcounter{letcnt1} 

\newcounter{letcnt2} 
\counterwithin{letcnt2}{letcnt1}

\newcounter{talkcnt} 

\def\a{\alpha}  
\def\b{\beta}  
\def\d{\delta}  
\def\e{\varepsilon}  

\def\z{\zeta}
\def\m{\mu}
\def\n{\nu}

\def\w{\omega} 

\def\x{\xi}
\def\GsD{\Gamma\backslash D}

\def\bC{\mathbb C} \def\cC{\mathcal C}

\def\td{\mathrm{d}} \def\sfd{\mathsf{d}}
 \def\tdet{\mathrm{det}}
 
 \def\tdim{\mathrm{dim}}

\def\tEnd{\mathrm{End}}

\def\cF{\mathcal F}

\def\tGr{\mathrm{Gr}}
\def\fg{{\mathfrak{g}}}

\def\bh{\mathbf{h}}
\def\bi{\mathbf{i}}

\def\sfm{\mathsf{m}}

 \def\sfn{\mathsf{n}}

 \def\cO{\mathcal O}

 \def\olP{{\overline P}}

\def\bQ{\mathbb Q}

\def\bR{\mathbb R}

\def\tRe{\mathrm{Re}}

\def\fs{\mathfrak{s}}
\def\bs{\mathbf{s}}

 \def\tspan{\mathrm{span}}

\def\sfw{\mathsf{w}}
  
\def\sX{\mathscr{X}}

\def\fz{\mathfrak{z}} 
 
\def\half{\tfrac{1}{2}}

\def\tand{\quad\hbox{and}\quad}
\def\bs{\backslash}

\def\smallb{{\hbox{\small{$\bullet$}}}}

\def\op{\oplus}
\def\ot{\otimes}

\def\tw{\hbox{\small $\bigwedge$}}

\newenvironment{a.list}
  {\begin{enumerate}[label=\alph*.,itemsep=3pt,leftmargin=25pt,listparindent=20pt]}
  {\end{enumerate}}

\newenvironment{num.list}
  {
  \begin{enumerate}[itemsep=3pt,leftmargin=25pt,listparindent=20pt,label={\arabic*.}]
  }
  {\end{enumerate}}
\newenvironment{i_list}
  {\begin{enumerate}[label=(\roman*),itemsep=3pt,leftmargin=25pt,listparindent=20pt]}
  {\end{enumerate}}
\newenvironment{i_list_emph}
  {\begin{enumerate}[label=\emph{(\roman*)},itemsep=3pt,leftmargin=25pt,listparindent=20pt]}
  {\end{enumerate}}
\newenvironment{I_list}
  {\begin{enumerate}[label=(\Roman*),itemsep=3pt,leftmargin=25pt,listparindent=20pt]}
  {\end{enumerate}}

\newtheorem{corollary}[equation]{Corollary}
\newtheorem{lemma}[equation]{Lemma}

\newtheorem{theorem}[equation]{Theorem}
\newtheorem{conjecture}[equation]{Conjecture}
\newtheorem*{theorem*}{Theorem}

\theoremstyle{definition}

\newtheorem*{boldQ*}{Question}
\newtheorem*{boldP*}{Problem}

\theoremstyle{definition}

\theoremstyle{remark}
\newtheorem*{assume*}{Assume}
\newtheorem*{answer*}{Answer}

\newtheorem*{claim*}{Claim}

\newtheorem*{definition*}{Definition}

\newtheorem*{example*}{Example}

\newtheorem*{hint*}{Hint}
\newtheorem*{notation*}{Notation}
\newtheorem{remark}[equation]{Remark}
\newtheorem*{remark*}{Remark}
\newtheorem*{remarks*}{Remarks}
\newtheorem*{fact*}{Fact}

\newtheorem*{emphQ*}{Question}
\newtheorem*{emphA*}{Answer}

\def\bs{\backslash}
\def\ddb{\partial\overline{\partial}}
\def\olt{\overline{t}{}}
\def\dbar{\overline{\partial}}
\def\tsum{\textstyle{\sum}}

\def\bfh{\mathbf{h}}
\def\sfh{\mathsf{h}}

\def\ola{\overline{\a}{}}
\def\olb{\overline{\b}{}}

\def\olv{\overline{v}}

\def\olx{\overline{\xi}{}}

\def\oleta{\overline\eta{}}

\def\dv{\partial_v}
\def\dbv{\dbar_{\olv}}
\def\ddbv{\dv \dbv}
\def\tdv{\td_v}
\def\tdbv{\td_{\olv}}

\def\olB{\overline B}

\def\GsD{\Gamma\bs D}

\def\PhiS{\Phi{}^\mathsf{S}}
\def\hPhiS{\hat\Phi{}^\mathsf{S}}

\def\hatP{\hat\wp}
\def\olP{\overline{\wp}}

\begin{document}
\title[Pseudoconvexity at infinity: a codimension one example]{Pseudoconvexity at infinity in Hodge theory: a codimension one example}



\author[Robles]{Colleen Robles}
\email{robles@math.duke.edu}
\address{Mathematics Department, Duke University, Box 90320, Durham, NC  27708-0320} 
\thanks{Robles is partially supported by NSF DMS 1611939, 1906352.}

\date{\today}

\begin{abstract}
The generalization of the Satake--Baily--Borel compactification to arbitrary period maps has been reduced to a certain extension problem on certain ``neighborhoods at infinity''.  Extension problems of this type require that the neighborhood be pseudoconvex.  The purpose of this note is to establish the desired pseudoconvexity in one relatively simple, but non-trivial, example: codimension one degenerations of a period map of weight two Hodge structures with $p_g=2$. 
\end{abstract}
\maketitle

\section{Introduction}

This paper is a companion note to \cite{Robles-extnHnorm}.  The work is motivated by the problem to generalize the Satake--Baily--Borel compactification to arbitrary period maps.  

\subsection{Motivation}

Briefly, the compactification problem is as follows.  Suppose that $D$ is a Mumford--Tate domain parameterizing pure, effective, weight $\sfw$, $Q$--polarized Hodge structures on a finite dimensional rational vector space $V$.  Fix a period map $\Phi : B \to \GsD$  defined on a smooth quasi-projective $B$ with smooth projective completion $\olB \supset B$ and simple normal crossing divisor $Z = \olB \bs B$ at infinity.  Let $\wp = \Phi(B)$ denote the image.  A proper \emph{topological} Satake--Baily--Borel (SBB) type completion $\PhiS : \olB \to \olP$ of the period map is constructed in \cite{GGRinfty}.  (See \cite{Robles-extnHnorm} for a more detailed review than we give here.)  Without loss of generality the period map is proper.  Then $\wp$ is quasi-projective \cite{BBT18}.  One would like to assert that $\olP$ is projective algebraic.  This is known to be the case when $D$ is hermitian and $\Gamma$ is arithmetic: $\olP$ is the closure of $\wp$ in the Satake--Baily--Borel compactification of $\Gamma\bs D$.  In general it is an open problem to show that $\olP$ is a complex analytic space.  

The completion $\PhiS$ is continuous, proper and admits a ``Stein factorization''
\[
\begin{tikzcd}
  \olB \arrow[r,"\hPhiS"]
  & \hatP \arrow[r]
  & \olP \,.
\end{tikzcd}
\]
The fibres of $\hatP \to \olP$ are finite, and the fibres of $\hPhiS$ are connected, compact algebraic subvarieties of $\olB$.  It is conjectured that the topological space $\hatP$ is Moishezon, and the map $\hPhiS : \olB \to \hatP$ is a morphism.  The conjecture holds in the case that $D$ is hermitian symmetric, and in the case that $\tdim\,\wp \le 2$, \cite{GGLR}.  Let $Z_1,\ldots,Z_\nu$ denote the smooth irreducible components of $Z$, and $Z_I = \cap_{i\in I}\, Z_i$ the closed strata.  The conjecture has been reduced to showing that every fibre $A$ of $\hPhiS$ admits a neighborhood $X \subset \olB$ with the following properties \cite[Theorem 3.20]{GGRinfty}:
\begin{I_list}
\item \label{i:proper}
The restriction of $\hPhiS$ to $X$ is proper.
\item \label{i:extn}
Holomorphic functions on $Z_I \cap X$ extend to $X$.
\end{I_list}

\noindent About property \ref{i:proper} we can say the following.  Let $\cF^{\sfw} \subset \cF^{\sfw-1} \subset \cdots \subset \cF^0$ denote the Hodge vector bundles over $B$.  Assume that the local monodromy at infinity is unipotent, so that the $\cF^p$ extend to $\olB$.  

\begin{theorem}[{\cite{GGRinfty}}] \label{T:ggr}
Every fibre $A$ of $\hPhiS$ admits a neighborhood $X \subset \olB$ with the following properties: 
\begin{i_list_emph}
\item 
The restriction of $\hPhiS$ to $X$ is proper.
\item \label{i:triv}
For every $p$, there exists a positive integer $1 \le m_p$ so that the line bundle $\tdet(\cF^p)^{\ot m_p}$ is trivial over $X$.
\end{i_list_emph}
\end{theorem}

The second property \ref{i:extn} is an Ohsawa--Takegoshi type extension problem (although without the need for bounds on the $L^2$ norms) \cite{MR1782659, MR3525916}.  Such theorems usually impose the hypothesis that $X$ is pseudoconvex.

\subsection{Pseudoconvexity}

Recall that the neighborhood $X$ is \emph{pseudoconvex} if it admits a plurisubharmonic exhaustion function $\rho : X \to \bR$.  A continuous function $\rho: X \to \bR$ is an \emph{exhaustion} if $\rho^{-1}[-\infty,r)$ is relatively compact for all $r \in \bR$.  The function is \emph{plurisubharmonic} (psh) if for every holomorphic map $\psi : \Delta \to X$, the composition $\rho \circ \psi$ is subharmonic.  If $\rho$ is $\cC^2$, then it is psh if and only if $\bi\ddb \rho \ge 0$.  For example, if $f \in \cO(X)$, then $\rho = |f|^2$ is psh.  Likewise, a line bundle with metric $h$ is positive if $-\log  h$ is psh.  Oka's Theorem asserts that a complex manifold is Stein if and only if it admits a smooth strictly psh exhaustion function. 

\begin{conjecture}[{\cite{Robles-extnHnorm}}] \label{conj:pseudocnvx}
The neighborhood $X$ in Theorem \ref{T:ggr} may be chosen to be pseudoconvex.  There is a continuous exhaustion function $\rho : X \to [0,\infty)$ with the property that $\ddb \rho( v , \overline v) \ge 0$, and equality holds if and only if $v$ is tangent to a fibre of $\PhiS$.
\end{conjecture}

\noindent The conjecture holds in the case that $D$ is hermitian symmetric, the case that $A \subset B$, and the case that $A$ is a connected component of $Z$, \cite{Robles-extnHnorm}.   The purpose of this note is to prove the conjecture in the following special, but nontrivial, case.  Let $Z_I^* = Z_I \bs \cup_{j\not\in I}Z_j$ denote the open strata of $\olB$.

\begin{theorem}\label{T:pseudocnvx}
Suppose that the Mumford--Tate domain $D$ parameterizes weight $\sfw=2$, effective, polarized Hodge structures with $p_g = h^{2,0} = 2$.  Assume that the fibre $A$ is contained in a codimension 1 strata $Z_i^*$.  Then Conjecture \ref{conj:pseudocnvx} holds.
\end{theorem}

\begin{remark}[Regularity]
When $D$ is hermitian, the exhaustion function will be smooth \cite{Robles-extnHnorm}.  In general, $\rho$ will be $\cC^1$, but not $\cC^2$.  Then the inequality $\ddb\rho(v,\overline v) \ge0$ of Conjecture \ref{conj:pseudocnvx} should be understood to allow $\ddb\rho(v,\overline v) =+\infty$.  The latter may arise when $v$ is normal $Z$, cf.~\S\S\ref{S:psh1}, \ref{S:psh2} and \ref{S:psh3}.  In this case $\bi\ddb\rho$ is a positive current.
\end{remark}

\begin{remark}[Strict psh] 
The exhaustion function $\rho : X \to \bR$ will be the $\hPhiS$--pullback of a continuous function $\varrho$ on
\[
  \sX \ = \ \hPhiS(X) \ \subset \ \hatP \,.
\]
The assertion that $\ddb \rho( v , \overline v) \ge 0$, with equality precisely when $v$ is tangent to a fibre of $\PhiS$, should be interpreted as saying that $\varrho$ is a \emph{strictly} psh function on $\sX$.  This is ``interpretative'' because the topological space $\hatP$ is not yet shown to be complex analytic.  What we can say is that the space $\hatP$ is a finite union $\cup\,\hatP_\pi$ of quasi-projective varieties \cite{GGRinfty}, and the restriction $\varrho\big|_{\hatP_\pi}$ is strictly psh.
\end{remark}

\begin{remark}[Related work]
Griffiths and Schmid showed that $D$ admits a smooth exhaustion function whose Levi form, restricted to the horizontal subbundle of the holomorphic tangent bundle, is positive definite at every point \cite[(8.1)]{MR0259958}.    In particular, the image of the lift $\tilde\Phi : \tilde B \to D$ to the universal cover of $B$ admits a strict psh exhaustion function.
\end{remark}

\section{Proof of Theorem \ref{T:pseudocnvx}: preliminaries}

\subsection{The basic idea} \label{S:prf-pcnvx}

Define
\[
  \Lambda \ = \ \tdet(\cF^\sfw) \,\ot\, \tdet(\cF^{\sfw-1}) 
  \,\ot\cdots
  \ot\, \tdet(\cF^{\lceil(\sfw+1)/2\rceil}) \,.
\]
Theorem \ref{T:ggr}\emph{\ref{i:triv}} implies $\Lambda^{\ot m}$ is trivial over $X$ for some positive integer $m\ge1$.

\begin{remark}
For simplicity, we will assume that $m=1$ in the proof of Theorem \ref{T:pseudocnvx}.
\end{remark}

\noindent  We will construct two functions $\rho_0 , \rho_1 : X \to \bR$ with the following properties:
\begin{i_list}
\item \label{i:rho0}
The function $\rho_0$ is psh on $X$, and vanishes along $Z \cap X$.
\item
The restriction $\rho_1\big|_{Z \cap X}$ is psh.
\item \label{i:psh}
The sum $\rho_0 + \rho_1 : X \to \bR$ is psh.  In fact, $\ddb (\rho_0+\rho_1)( v , \overline v) \ge 0$, with equality precisely when $v$ is tangent to a fibre of $\PhiS$.
\item \label{i:A}
We have $\rho_0 + \rho_1 \ge 0$, and the fibre is characterized by 
\[
	A \ = \ \{ \rho_0 + \rho_1 \ = \ 0 \} \,.
\]
\end{i_list}
Then for sufficiently small $\e>0$ we may take $X = \{ \rho_0+\rho_1 < \e\}$ and $\rho = 1/(\e-\rho_0-\rho_1)$.

There are (at least) two possibilities for the function $\rho_0$.  The triviality of $\Lambda\big|_X$ and \cite[Theorem 6.14]{BBT18} yields holomorphic functions $g_1,\ldots,g_\m \in \cO(X)$ that separate the fibres of $\Phi\big|_{B \cap X}$ and have the property that $V(g_1,\ldots,g_\m) = Z \cap X$.  In particular, the psh function
\begin{equation} \label{E:rho0a}
  \rho_0 \ = \ |g_1|^2 \,+\,\cdots\,+\, |g_\m|^2
\end{equation}
descends to $\sX = \hPhiS(X)$, is strictly psh on the complex analytic variety 
\[
  \sX_0 \ = \ \hPhiS(B \cap X)
\]
(which is dense in $\sX$), and vanishes along $\hPhiS(Z \cap X) = \sX \bs \sX_0$.  This choice of $\rho_0$ will work well when $\rho_1$ is psh on all of $X$.  (This includes the case the $D$ is hermitian \cite{Robles-extnHnorm}, and a few others, including the two in \S\ref{S:lastdeg}.)

In general it appears that the $\rho_1$ constructed here will not be psh on all of $X$ (\S\S\ref{S:mindeg}-\ref{S:deg3}).  In this case, we do not know enough about the vanishing of the $g_j$ along $Z \cap X$ to conclude that $\rho_0 + \rho_1$ is psh.  Instead we let $h_0$ be the Hodge norm-squared of a trivialization of $\Lambda\big|_{B\cap X}$.  Then
\begin{equation}\label{E:rho0b}
   \rho_0 \ = \ 1/h_0 \,:\, X \ \to \ \bR
\end{equation}
is psh, non-negative with vanishing locus $\{ \rho_0 = 0 \} = Z \cap X$, and descends to a strictly psh function on $\sX_0$ that vanishes along $\sX \bs \sX_0$.  The advantage of \eqref{E:rho0b} over \eqref{E:rho0a} is that the asymptotic behavior of the former is very well understood, and we will be able to show that $\rho_0 + \rho_1$ is psh.  (The advantage of \eqref{E:rho0a} is that it is smooth on all of $X$.  The function \eqref{E:rho0b} will be continuous, but not smooth in general, cf.~\S\S\ref{S:h1}, \ref{S:h2}, \ref{S:h3}.)

The function $\rho_1$ will be realized as $-\log h$ with $h : X \to \bR$ an extension of the Hodge norm on $Z \cap X$.  

\subsection{Extension of Hodge norms}
There is a Hodge metric associated to each $\Lambda\big|_{Z_I^*}$ that is canonically defined up to a positive multiple.  Fix a trivialization of $\Lambda\big|_X$ and let $h_I : Z_I^* \cap X \to \bR_{>0}$ be the Hodge norm-squared of the trivialization.  Then $-\log h_I$ is a smooth psh function \cite{MR0259958}.

\begin{theorem}[{\cite{Robles-extnHnorm}}] \label{T:h}
The neighborhood $X$ of Theorem \ref{T:ggr} may be chosen so that it admits a continuous function $h : X \to \bR$ that is smooth on strata $Z_I^* \cap X$ \emph{(}including $B \cap X$\emph{)}, constant on $\hPhiS$--fibres, and has the following property: if $Z_I^* \cap A$ is nonempty, then the restriction of $h$ to $Z_I^*$ is a multiple of the Hodge norm-squared $h_I$.  In particular the restriction of $-\log h$ to $Z_I^*$ is plurisubharmonic.
\end{theorem}

\begin{remark}
If the Mumford--Tate domain $D$ is hermitian, then $h$ is smooth and $-\log h$ is psh \cite{Robles-extnHnorm}.  For an arbitrary Mumford--Tate domain, both smoothness and plurisubharmonicity may fail, cf.~\S\S\ref{S:mindeg}-\ref{S:deg3}.  (Although, the restriction $h\big|_{Z_J^* \cap X}$ is always smooth for all strata $Z_J^*$.)
\end{remark}

\subsection{Weight $\sfw=2$ period domain with $p_g=2$} \label{S:wt2}

Our proof of Theorem \ref{T:pseudocnvx} will make liberal use of the companion paper \cite{Robles-extnHnorm}, especially \S\S\ref{extn-S:prelim}--\ref{extn-S:prfTh} of that paper, and we follow the notation there. 

Let $D$ be the period domain parameterizing weight $\sfw=2$, polarized Hodge structures on a rational vector space $V$ with Hodge numbers $\bfh(V) = (2,\sfh,2)$.  The induced Hodge structure on 
\[
  H \ = \ \tw^2 V
\]
has weight $\sfn=2$ and Hodge numbers $\bfh(H) = (1 \,,\, 2 \sfh \,,\, \half \sfh(\sfh-1)+4 \,,\, 2\sfh \,,\, 1)$.  We assume that the fibre $A$ is contained in a codimension one strata $Z_i^* \subset Z$; without loss of generality $i=1$ and 
\[
  A \subset Z_1^* \,.
\]

Let $\rho_1 = -\log h$, with $h$ given by Theorem \ref{T:h}.  The problem is to show that we may choose a psh $\rho_0$ so that the conditions \ref{i:psh} and \ref{i:A} of \S\ref{S:prf-pcnvx} are satisfied.  As discussed in \S\ref{S:prf-pcnvx} our choice of $\rho_0$, will depend on properties of $\rho_1$.  Both $\rho_0$ and $\rho_1$ are defined in  terms of the period matrix representation (\cite[\S\ref{extn-S:pmr}]{Robles-extnHnorm}) over $X$ and so are constant on fibres of $\hPhiS$.  It suffices verify \S\ref{S:prf-pcnvx}\ref{i:psh}-\ref{i:A} in a local coordinate chart $(U,t)$ centered at a point $b \in A$. (Here we follow the notation of \cite[\S\ref{extn-S:loc1}]{Robles-extnHnorm}.)  Without loss of generality the coordinates satisfy 
\[
  U \,\cap\, Z_1 \ = \ \{ t_1=0\} \,.
\]

There are are five possible ``types'' for the limiting mixed Hodge structure $(W,F,N)$ arising along $U \cap Z_1$ \cite{MR4012553, MR3701983}.  The proof of Theorem \ref{T:pseudocnvx} is explicit case-by-case analysis, one for each type, in \S\S\ref{S:mindeg}--\ref{S:lastdeg}, respectively.  These types are indexed by the Hodge numbers $\bfh_\ell$ of the pure, weight $\ell$ Hodge structure on $\tGr^W_\ell(V)$ determined by $F$.  It is convenient to visually represent each of these types by the associated Hodge diamond (\S\ref{S:hd}).

In general, the mixed Hodge structure $(W,F)$ is not $\bR$--split.  It will be convenient in the computations that follow to assume that we have replaced $(W,F)$ with an $\bR$--split mixed Hodge structure $(W,\tilde F)$ given by $\tilde F = \exp(\d) \cdot F$ with $\d \in \op_{p,q\le-1}\,\fg^{p,q}_{W,F}$.  (This may always been done, \cite{MR840721}.)  After this replacement, it will still be the case that the assertions of \cite[\S\ref{extn-S:Aloc}]{Robles-extnHnorm} hold.

Throughout we will make frequent use of the $Q$--isotropy of $W$
\begin{subequations} \label{SE:Wn=2}
\begin{eqnarray} 
  Q \left( W_\ell(V) \,,\, W_m(V) \right) & = &  0 \,,
  \quad \forall \ \ell+m < 4 \,,  \\
  Q \left( W_\ell(H) \,,\, W_m(H) \right) & = &  0 \,,
  \quad \forall \ \ell+m < 8 \,;
\end{eqnarray} 
as well as
\begin{equation}
  N(W_\ell) \ \subset \  W_{\ell-2} \,.
\end{equation}
\end{subequations}

\section{Proof of Theorem \ref{T:pseudocnvx} for the minimal degeneration} \label{S:mindeg}

A minimal degeneration $(W,F,N)$ has Hodge numbers $\bh_0 = (0)$, $\bh_1  =  (1,1)$ and  $\bh_2 =  (1,\sfh-2,1)$.  The associated Hodge diamond $\Diamond(V)$ is
\begin{equation} \label{E:hd1}
\begin{tikzpicture}[baseline=(current  bounding  box.center)]
  \node [above] at (1,2.2) {};
  \draw [<->] (0,2.5) -- (0,0) -- (2.5,0);
  \draw [gray] (1,0) -- (1,2);
  \draw [gray] (2,0) -- (2,2);
  \draw [gray] (0,1) -- (2,1);
  \draw [gray] (0,2) -- (2,2);
  \draw [fill] (0,2) circle [radius=0.07];
  \node [left] at (0,2) {$v_6$};
  \draw [fill] (0,1) circle [radius=0.07];
  \node [left] at (0,1) {$v_5$};
  \draw [fill] (1,2) circle [radius=0.07];
  \node [above] at (1,2) {$v_3$};
  \draw [fill] (1,1) circle [radius=0.07];
  \node [above right] at (1,1) {$v_r$};
  \draw [fill] (1,0) circle [radius=0.07];
  \node [below] at (1,0) {$v_4$};
  \draw [fill] (2,1) circle [radius=0.07];
  \node [right] at (2,1) {$v_2$};
  \draw [fill] (2,0) circle [radius=0.07];
  \node [below] at (2,0) {$v_1$};
\end{tikzpicture}
\end{equation}
See \eqref{E:hd2-1} for the Hodge diamond $\Diamond(H)$.

We may choose a basis of $\{v_1,\ldots,v_\sfd\}$ of $V_\bC$ so that:  the polarization satisfies $Q(v_a,v_b) = \d^7_{a+b}$ for all $1 \le a,b\le 6$, $Q(v_r,v_s) = \d_{rs}$ for all $7 \le r,s$, and all other pairings are zero; the underlying real structure is $\overline{v_1} = -v_6$, $\overline{v_2} = v_3$, $\overline{v_4} = v_5$ and $\overline{v_r} = v_r$; the nilpotent operator is given by
\begin{equation} \label{E:Nsp}
  N \ = \ \bi(v_4 \ot v^2 - v_5 \ot v^3) \ \in \ \tEnd(V_\bR,Q) \,;
\end{equation}
the Hodge filtration is $F^2(V_\bC) = \tspan_\bC\{v_1,v_2\}$, and the weight filtration is $W_1(V_\bC) = \tspan_\bC\{ v_4 , v_5 \}$.  We have $e_0 = v_1 \wedge v_2$ and $e_\infty = v_1\wedge v_4$, in the notation of \cite[\S\S\ref{extn-S:pmr}-\ref{extn-S:einfty}]{Robles-extnHnorm}.
 
\subsection{Period matrix representation}

As discussed in \cite[\S\ref{extn-S:pmr}]{Robles-extnHnorm}, over $B \cap X$ we have $F^2_\Phi = \tspan_\bC\{ \x_1 \,,\, \x_2 \}$ where 
\[
  \x_a \ = \ v_a \ + \ \sum_{j\ge 3} \x_a^j\,v_j \,,\quad a=1,2 \,.
\]
The $\x_a^j : B \cap X \to \bC$ are holomorphic and defined up to the action of the monodromy $\Gamma_X$.  The
\begin{eqnarray*}
  \x_3 & = & v_3 \,-\, \x_2^4\,v_5 \,-\, \x_1^4\,v_6 \,,\qquad
  \x_4 \ = \ v_4 \,-\, \x_2^3\,v_5 \,-\, \x_1^3\,v_6 \\
  \x_r & = & v_r \,-\, \x_2^r\,v_5 \,-\, \x_1^r\,v_6 
  \,,\quad r \ge 7 \,.
\end{eqnarray*}
extend $\{\x_1,\x_2\}$ to a framing of $F^1(\Phi)$. 

By \cite[\eqref{extn-SE:A}]{Robles-extnHnorm}, the fibre $A \subset Z$ is cut out by 
\begin{equation}\label{E:A1}
  A \ = \ \{ \x^3_1 , \x^3_2 , \x^r_1 , \x^6_1 = 0 \} \,.
\end{equation} 
We have $\eta_0 = \x_1 \wedge \x_2$ and $\eta_\infty = \x_1\wedge \x_4$, in the notation of \cite[\S\S\ref{extn-S:pmr}, \ref{extn-S:dfnh}]{Robles-extnHnorm}.

\subsection{Matrix coefficients in local coordinates}\label{S:matcoef1}

As discussed in \cite[\S\ref{extn-S:loc1}]{Robles-extnHnorm}, we have $\x_j = \exp(\ell(t_1)N) \z(t)\cdot(v_j)$, with $\z : U \to \exp(\fs_F^\perp)$ holomorphic.  We have $\fs_F^\perp = (\fs_F^\perp \cap \fz_N) \op \bC(v_3 \ot v_1 -  v_6 \ot v^4)$, with $\fz_N$ the centralizer of $N$ in $\fg_\bC$.  We may factor $\z = \exp \x^3_1(v_3 \ot v^1 -  v_6 \ot v^4) \cdot \a$ with $\a : U \to \exp(\fs_F^\perp \cap \fz_N)$ holomorphic.  Set $\a_j = \a(v_j)$, and define $\a_j^i \in \cO(U)$ by $\a_j = \a_j^i\,v_i$.  We have $\x_4^2 = \a^4_2 + \bi\ell(t)$.  After a change of coordinates $t_1 \mapsto \exp(2\pi\a^4_2)\,t_1$, we have $\a^4_2 = 0$.  Then
\begin{eqnarray*}
  \x_1 & = & \a_1 \,+\, \x^3_1\, \left( 
  	\b_{1,3} \,+\, \ell(t_1) N \b_{1,3} \right) \,,\\
  \x_2 & = & \a_2 \,+\, \ell(t_1) N \a_2 \,,\qquad
  \x_4 \ = \ -\bi N \a_2 \,-\, \x^3_1\,v_6 \,,
\end{eqnarray*}
with 
\[
  \b_{1,3} = (v_3 - \a^4_1\,v_6) : U \to W_3(V_\bC)
\]
holomorphic.  The condition \cite[\eqref{extn-E:locW}]{Robles-extnHnorm} is equivalent to the vanishing of $\x^3_1$ along $U \cap Z = \{t_1=0\}$.  In particular,
\begin{equation}\label{E:n31}
  \xi^3_1 \ = \ t_1\,\n^3_1
\end{equation}
for some holomorphic $\n^3_1 : U \to \bC$.  

\subsection{The sections $\eta_0$ and $\eta_\infty$ in local coordinates}

Set 
\begin{eqnarray*}
  \a_{0,5} & = & \a(v_1\wedge v_2) : U \ \to \ W_5(H_\bC) \\
  \b_{0,6} & = & \b_{1,3} \wedge \a_2 : U \ \to \ W_6(H_\bC) \\
  \b_{\infty,5} & = & v_6 \wedge \b_{1,3}
  \ = \ v_3 \wedge v_6: U \ \to \ W_5(H_\bC) \\
  \b_{\infty,4} & = & -\bi\,
  \b_{1,3}\wedge( N \a_2) \,-\, \a_1 \wedge v_6 : 
  U \ \to \ W_4(H_\bC) \,.
\end{eqnarray*}
Note that $N^3\b_{0,6}, N^2 \b_{\infty,4} , N^2\b_{\infty,5}  = 0$, $\b_{0,6} \equiv v_3 \wedge  v_2$ modulo $W_5(H_\bC)$, and  $N^2\b_{0,6} = 2\bi N \b_{\infty,4} = 2 v_5 \wedge v_4$.  We have
\begin{eqnarray*}
  \eta_0
  & = & \a_{0,5} \,+\, \ell(t_1) N \a_{0,5} \ + \ 
  \x^3_1 \left(
    \b_{0,6} \,+\, \ell(t_1) N \b_{0,6} \,+\, 
    \half \ell(t_1)^2 N^2 \b_{0,6}
  \right) \\
  \eta_\infty
  & = & 
	-\bi N\a_{0,5} \,+\, 
	\x^3_1 \left( 
		\b_{\infty,4} \,+\, \ell(t) N \b_{\infty,4}
	\right) \,+\,
	(\x^3_1)^2 \left(
		\b_{\infty,5} \,+\, \ell(t) N \b_{\infty,5}
	\right) \,.
\end{eqnarray*}

\subsection{The Hodge norms in local coordinates} \label{S:h1}

The extension of the Hodge norm on $Z \cap X$ to $X$ is $h = -\tRe\,Q(\eta_0 , \oleta_\infty)$, cf.~\cite{Robles-extnHnorm}.  Keeping \eqref{SE:Wn=2} in mind, in local coordinates we have
\begin{eqnarray*}
  h & = & 
  - \bi Q(\a_{0,5} , N \ola_{0,5}) 
  \,+\, \tfrac{1}{2\pi}\,|\x^3_1|^2 \, \log |t_1|^2 
  \\ & & + \  
  \bi \, \x^3_1 \, 
  Q(\ola_{0,5} \,,\, N \b_{0,6} )
  \ + \  
  \half \, \x^3_1 \, 
  Q(\ola_{0,5} \,,\, \a_1\wedge v_6 - v_5  \wedge \a_2 ) 
  \\ & & - \  
  \bi \, \olx^3_1 \, 
  Q(\a_{0,5} \,,\, N \olb_{0,6} )
  \ + \  
  \half \,\olx^3_1 \, 
  Q(\a_{0,5} \,,\, 
  	\ola_1\wedge \olv_6 - \olv_5  \wedge \ola_2 ) 
  \\ & & 
  - \ \half \, |\x^3_1|^2 
  \left\{ Q(\b_{0,6} , \olb_{\infty,4})
  	\,+\, Q( \olb_{0,6} , \b_{\infty,4} ) \right\}
  \\ & & - \ 
  \half\,(\x^3_1)^2\,
  Q\left(\ola_{0,5} + \olx^3_1\,\olb_{0,6} \,,\, 
  	\b_{\infty,5} - 
	\tfrac{\bi}{2\pi} \log |t|^2 \, N\b_{\infty,5}
  \right)
  \\ & & - \ 
  \half\,(\olx^3_1)^2\,
  Q\left(\a_{0,5} + \x^3_1\,\b_{0,6} \,,\, 
  	\olb_{\infty,5} + 
	\tfrac{\bi}{2\pi} \log |t|^2 \, N\olb_{\infty,5}
  \right) \,.
\end{eqnarray*}

The Hodge norm on $B \cap X$ is $h_0 = Q(\eta_0,\oleta_0)$.  Again, keeping \eqref{SE:Wn=2} in mind, in local coordinates we have
\begin{eqnarray*}
  h_0 & = & 
  \tfrac{\bi}{2\pi}\,\log |t_1|^2 \, Q(\a_{0,5} , N\ola_{0,5})
  \ - \ \tfrac{1}{4\pi^2}\,(\log |t_1|^2)^2 \, |\x^3_1|^2 
  \\ & & + \
  \tfrac{\bi}{2\pi} \, \log |t_1|^2 \,
  \Big(
    \olx^3_1 \, Q(\a_{0,5} , N\olb_{0,6})
    \,-\, \x^3_1 \, Q(\ola_{0,5} , N \b_{0,6})
  \Big)
  \\ & & 
  + \ Q \Big( \a_{0,5} + \x^3_1\,\b_{0,6} \,,\,
  	\overline{\a_{0,5} + \x^3_1\,\b_{0,6}} \Big)
  \ + \ 	
  \tfrac{\bi}{2\pi} \, \log |t_1|^2 \, |\x^3_1|^2 \,
  	Q(\b_{0,6} , N\olb_{0,6}) \,.
\end{eqnarray*}

\subsection{Plurisubharmonicity} \label{S:psh1}

Set $\rho_0 = 1/h_0$ and $\rho_1 = -\log h$. We know that $\rho_0$ is psh on $X$, and that $\rho_1$ is psh on $Z \cap X$.  The goal of this section is to show that $\rho_0+\rho_1$ is psh on $X$.  Because plurisubharmonicity is a local property, it suffices to show $\rho_0+\rho_1$ is psh on $U$.  We compute 
\begin{eqnarray*}
  \ddb (\rho_0 + \rho_1) & = & 
  \ddb (1/h_0 - \log h) \\
  & = &
  \frac{-h_0 \ddb h_0 \,+\, 
  		2\, \partial h_0 \wedge \dbar h_0}{h_0^3} 
  \ + \ 
  \frac{-h\ddb h \,+\, \partial h \wedge \dbar h}{h^2} \,.
\end{eqnarray*}

\begin{lemma} 
Fix a holomorphic vector field $v$ on $U$ with $\tdv t_1 = 1$.  We have
\[
  \lim_{t_1\to0} \ddbv \, (\rho_0 + \rho_1) 
  \ = \ +\infty \,.
\]
\end{lemma}

\begin{proof}
The lemma follows from \eqref{E:n31} and the calculations in \S\ref{S:h1}.  The key point is that as $t_1 \to 0$, the expression $\ddbv\, (\rho_0 + \rho_1)$ is dominated by the term 
\[
  -\frac{Q(\a_{0,5} , N\ola_{0,5})^2}
  	{2\pi^2\,h_0^3\,|t_1|^2}
	\ = \ 
	-\frac{Q(\a_{0,5} , N\ola_{0,5})^2}
  	{2\pi^2\,h_0^3\,|t_1|^2}  
  	\,\tdv t_1 \wedge \tdbv \,\olt_1 
  \quad\hbox{in}\quad 
  \frac{2\,\dv h_0 \wedge \dbv \,h_0}{h_0^3} \,.
\]
Recall that $-\bi Q(\a_{0,5} , N\ola_{0,5})$ is the Hodge norm along $Z_1 \cap U$; in particular, $-\bi Q(\a_{0,5} , N\ola_{0,5}) > 0$.  Alternatively, one may directly compute
\begin{equation}\label{E:Qa1}
  -\bi Q(\a_{0,5} , N\ola_{0,5}) \ = \ 
  (1 + |\a^6_1|^2 - \tsum_r |\a^r_1|^2) (1 - |\a^3_2|^2) \,.
\end{equation}
By \eqref{E:A1} 
\begin{equation}\label{E:locA1}
  A \,\cap\, U \ = \ 
  \{ t_1\,,\ \a_1^r\,, \a^6_1\,,\ \a^3_2 \ = \ 0 \} \,.
\end{equation}
So \eqref{E:Qa1} is identically 1 along the fibre $A \cap U$.
\end{proof}

\begin{lemma} \label{L:psh1fnt}
Fix a holomorphic vector field $v$ on $U$ with $\tdv t_1 = 0$.  Shrinking $U$ if necessary, we have $\ddbv\, (\rho_0 + \rho_1) \ge 0$.
\end{lemma}

\begin{corollary} \label{C:phs1}
The function $\rho_0 + \rho_1$ is psh on $X$.
\end{corollary}

\subsection{Proof of Lemma \ref{L:psh1fnt}}

\subsubsection{Preliminaries} \label{S:prf1prelim}

The proof of the lemma is by lengthy analysis of the asymptotic behavior of $\rho_0$, $\rho_1$ and their derivatives.   Keeping \eqref{E:n31} in mind, we will locally regard both $h$ and $h_0$ as polynomials in $t_1 ,\, \olt_1 ,\, \log|t_1|^2$ with coefficients in the space $\cC^\w(U)$ of real-analytic functions on $U$; that is, 
\[
  h, h_0 \ \in \ \cC^\w(U)[t_1,\olt_1,\log |t_1|^2] \,.
\]  
Likewise, we will regard $\dv \dbv (\rho_0 + \rho_1) = \ddb (\rho_0+\rho_1)(v,\olv)$ as an element of the field of fractions of $\cC^\w(U)[t_1,\olt_1,\log |t_1|^2]$. 

Set
\[
  q_1(\a) \ = \ -\bi Q(\a_{0,5} , N\ola_{0,5} ) 
  \tand q_0(\a) \ = \ Q(\a_{0,5} , \ola_{0,5}) \,.
\]
From \eqref{E:n31} and \S\ref{S:h1} we see that 
\begin{eqnarray*}
  h & = & q_1(\a) \ + \ O(|t_1| \log |t_1|^2) \\
  h_0 & = & -\frac{1}{2\pi} \log |t_1|^2 \, q_1(\a)
  \ + \ q_0(\a) \ + \ O(|t_1| \log |t_1|^2)
\end{eqnarray*}
As observed following \eqref{E:Qa1}, $q_1(\a)$ is identically $1$ along the fibre $A \cap U$.  Shrinking $U$ if necessary, we may assume that $q_1(\a) > \half$ on $U$.  Then
\begin{eqnarray}
	\nonumber
  \ddbv\,\rho_1 & = & - \ddbv\,\log q_1(\a) 
  \ + \ O(|t_1| \log |t_1|^2) \\
  	\nonumber
  \ddbv\,\rho_0 & = & 
  \frac{(\log|t_1|^2)^2}{4 \pi^2  \, h_0^3} 
  \cdot  q_1(\a)^3 \, \ddbv \, (1/q_1(\a)) \\
  	\label{E:ddbrho01}
  & & + \ \frac{\log|t_1|^2}{2\pi \, h_0^3} \cdot
  \left\{q_1(\a)\,\ddbv\,q_0(\a) \,+\, 
  	q_0(\a)\,\ddbv\,q_1(\a) \right\} \\
	\nonumber
  & & - \ \frac{\log|t_1|^2}{\pi \, h_0^3} \cdot
  \left\{
    \dv q_1(\a)\,\dbv\,q_0(\a) \,+\,
    \dv q_0(\a)\,\dbv\,q_1(\a)
  \right\} \\
  	\nonumber
  & & + \ \frac{O(|t_1|\,(\log|t_1|^2)^2)}{h_0^3} \,.
\end{eqnarray}
In particular
\begin{eqnarray*}
  \lim_{t_1\to0} \ddbv\, \rho_0 & = & 0 \,,\\
  \lim_{t_1\to0} \ddbv\, \rho_1 & = & 
  -\ddbv\,\log q_1(\a) \,,
\end{eqnarray*}
(All limits are taken with $t_2,\ldots,t_r$ constant.)  

\subsubsection{Step 1} \label{S:s1}

It may be seen directly from \eqref{E:Qa1}, \eqref{E:locA1} and the first Hodge--Riemann bilinear relation $0 = Q(\a_1,\a_1) = 2\a^6_1 \,+\, \sum_r (\a^r_1)^2$ that $-\log q_1(\a)$ is psh in a neighborhood of $A \cap U$; without loss of generality, $-\log q_1(\a)$ is psh on $U$.  If 
\[
  \lim_{t_1\to0} -\ddbv\,\log q_1(\a) \ > \ 0 \,,
\]
then it follows from \S\ref{S:prf1prelim} that, after shrinking $U$ if necessary, we have $\ddbv\,(\rho_0+\rho_1) > 0$ on $U$.

\subsubsection{Step 2} \label{S:s2}

Now suppose that
\begin{equation}\label{E:limdQaNa1}
  \lim_{t_1\to0} - \,
  \ddbv \,\log q_1(\a) \ = \ 0 \,.
\end{equation}
Equivalently,
\begin{equation} \nonumber 
  \lim_{t_1 \to 0}
  \tdv \a^r_1 \,,\, \tdv \a^6_1 \,,\, \tdv \a^3_2
  \ = \ 0 \,.
\end{equation}
Consider the coefficient of $\log|t_1|^2$ in the expression for $\ddbv\,\rho_0$ in \eqref{E:ddbrho01}.  
We compute
\begin{eqnarray*}
  \lim_{t_1\to0} 
  \ddbv\,q_0(\a)
  & = &
  Q(\a_1,\ola_1) \, 
  Q \big( \tdv \a_2 , \tdbv\, \ola_2 \big) \\
  & & - \ 
  \left| Q \big(\a_2 , \tdbv\,\ola_1 \big) \right|^2
  \ - \ 
  \left| Q \big(\a_1 , \tdbv\,\ola_2 \big) \right|^2 
  \\
  & = & 
  -( 1 + |\a^6_1|^2 - \tsum_r|\a^r_1|^2)
  \left[ \tsum_r \left|\tdv  \a^r_2 \right|^2 \,-\,
  \left|\tdv \a^6_2 \right|^2  \right] \\
  & & - \
  \left| \tdv \a^4_1 \,+\,\ola^3_2\,\tdv \a^5_1 \right|^2
  \ - \ 
  \left| \tsum_r \ola^r_1\,\tdv \a^r_2 \,-\,
  	\ola^6_1 \tdv \a^6_2 \right|^2 \,.
\end{eqnarray*}
The infinitesimal period relation $0 = Q(\x_1,\td\x_2) = Q(\td \x_1 , \x_2)$ yields $\tdv \a^6_2 = - \sum_r \a^r_1\tdv \a^r_2$ and $\tdv \a^5_1 = -\a^3_2\tdv \a^4_1$, along $t_1=0$.  It follows that
\begin{equation}\label{E:limdQaa1}
  \lim_{t_1\to0} 
  \ddbv \, q_0(\a) \ \le \ 0
\end{equation}
with equality if and only if $\td\a^4_1(v) \,,\ \td\a^r_2(v) =0$ at $t_1=0$.  If the inequality \eqref{E:limdQaa1} is strict,  then $\ddbv (\rho_0+\rho_1)$ is asymptotically dominated by the $\log|t_1|^2$ terms in \eqref{E:ddbrho01}.  In particular, $\ddb (\rho_0+\rho_1)(v,\olv) > 0$ for $0 < |t_1|$ sufficiently small.

\subsubsection{Step 3} \label{S:s3}

We continue to assume that \eqref{E:limdQaNa1}, and now suppose that equality holds in \eqref{E:limdQaa1}.  These conditions are equivalent to 
\begin{equation}\label{E:limda1}
  \lim_{t_1\to0} \tdv \a_j^i \,,\qquad \forall  \ i,j  \,.
\end{equation}
This means that all differentials $\td_v\,\x^i_j$ of the period matrix representation vanish along $\{t_1=0\}$.  If it is the case that $\tdv \n^3_1 \not=0$ at $t_1=0$, then $-\ddbv(\rho_0+\rho_1)$ is asymptotically dominated by the terms
\[
  |t_1|^2\,\log |t_1|^2
  \left( \frac{\log |t_1|^2}{4\pi^2\,h_0^2}
  	\,-\, \frac{1}{2\pi\,h}  \right) \, 
  \left| \tdv \n^3_1 \right|^2 \,.
\]
We have $-\ddbv(\rho_0+\rho_1) > 0$ for $0 < |t_1|$ sufficiently small.

\subsubsection{Step 4}

It remains only to consider the case that \eqref{E:limda1} holds, and $\tdv \n^3_1$ vanishes along $t_1=0$.  In this case, the formulae of \S\S\ref{S:matcoef1}--\ref{S:h1} imply that  
\begin{eqnarray*}
  \ddbv\,(\rho_0+\rho_1) & = & 
  - \ddbv\,\log q_1(\a) \\
  & & + \ \frac{\log|t_1|^2}{2\pi \, h_0^3} \cdot
  \left\{q_1(\a)\,\ddbv\,q_0(\a) \,+\, 
  	q_0(\a)\,\ddbv\,q_1(\a) \right\} \\
	\nonumber
  & & - \ \frac{\log|t_1|^2}{\pi \, h_0^3} \cdot
  \left\{
    \dv q_1(\a)\,\dbv\,q_0(\a) \,+\,
    \dv q_0(\a)\,\dbv\,q_1(\a)
  \right\} \\
  & & + \ 
  |t_1|^2\,\log |t_1|^2
  \left( \frac{\log |t_1|^2}{4\pi^2\,h_0^2}
  	\,-\, \frac{1}{2\pi\,h}  \right) \, 
  \left| \tdv \n^3_1 \right|^2 
  \ + \ \cdots 
\end{eqnarray*}
is asymptotically dominated by one of the terms considered in \S\S\ref{S:s1}--\ref{S:s3} (the terms displayed here).  We conclude that $\ddbv(\rho_0+\rho_1) \ge 0$ on $U$ (possibly after shrinking $U$).  This completes the proof of Lemma \ref{L:psh1fnt}. \hfill\qed

\subsection{Pseudoconvexity}

In order to complete the proof of Theorem \ref{T:pseudocnvx} for the degeneration \eqref{E:hd1} as outlined in \S\ref{S:prf-pcnvx} it suffices to observe that: (1) the analysis in the proof of Lemma \ref{L:psh1fnt} implies \S\ref{S:prf-pcnvx}\ref{i:psh}; and (2) we have $A \subset \{ \rho_0  + \rho_1\}$, and an analysis similar to that in the proof of Lemma \ref{L:psh1fnt} implies \S\ref{S:prf-pcnvx}\ref{i:A}.
\hfill\qed

\bigskip

The arguments for the remaining four degenerations are similar to those above, and we will be increasingly brief as we work through in them in \S\S\ref{S:deg2}--\ref{S:lastdeg}.  We continue to make liberal use of \cite[\S\S\ref{extn-S:prelim}-\ref{extn-S:prfTh}]{Robles-extnHnorm}.

\section{Proof of Theorem \ref{T:pseudocnvx} for the second degeneration} \label{S:deg2}

The next degeneration has limiting mixed Hodge structure $(W,F,N)$ satisfying $W_0(V) = W_1(V)$ and $W_2(V) = W_3(V)$, and with Hodge numbers $\bfh_2 = (1,\sfh,1)$ and $\bfh_0 = (1)$.  The associated Hodge diamond $\Diamond(V)$ is 
\begin{equation}\label{E:hd2}
\begin{tikzpicture}[baseline={([yshift=-.5ex]current bounding box.center)}]
  \draw [<->] (0,2.75) -- (0,0) -- (2.75,0);		
  \draw [gray] (1,0) -- (1,2);
  \draw [gray] (2,0) -- (2,2);
  \draw [gray] (0,1) -- (2,1);
  \draw [gray] (0,2) -- (2,2);
  \draw [fill] (0,2) circle [radius=0.07];
  \node [above right] at (0,2) {{$v_5$}};
  \draw [fill] (0,0) circle [radius=0.07];
  \node [above right] at (0,0) {{$v_4$}};
  \draw [fill] (1,1) circle [radius=0.07];
  \node [above right] at (1,1) {{$v_3$}};
  \node [below right] at (1,1) {{$v_r$}};
  \draw [fill] (2,2) circle [radius=0.07];
  \node [above right] at (2,2) {{$v_2$}};
  \draw [fill] (2,0) circle [radius=0.07];
  \node [above right] at (2,0) {{$v_1$}};
\end{tikzpicture}
\end{equation}
See \eqref{E:hd2-2} for the Hodge diamond on $\Diamond(H)$.

We may choose the basis $\{v_1,\ldots,v_\sfd\}$ of $V_\bC$ so that: the polarization satisfies $Q(e_a,e_b) = \d^6_{a+b}$ for all $a,b \le 5$, and $Q(e_r,e_s) = \d_{rs}$ for all $r,s\ge6$, and all other pairings are zero; the Hodge filtration is given by  $F^2 = \tspan\{v_1,v_2\}$ and $F^1 = \tspan \{ v_1,v_2,v_3 , v_r\}_{r\ge 6}$; the underlying real structure is $\overline{v_1} = -v_5$, $\overline{v_3}=-v_3$ and all other basis vectors are real; the nilpotent operator is given by $N = \bi(v_4 \ot v^3-v_3 \ot v^2)$; and the Hodge filtration is $F^2(V) = \tspan\{v_1,v_2\}$.  We have $e_0 = v_1 \wedge v_2$ and $e_\infty = v_1\wedge v_4$, in the notation of \cite[\S\S\ref{extn-S:pmr}-\ref{extn-S:einfty}]{Robles-extnHnorm}.

\subsection{Period matrix representation}

Over $B \cap X$ we have $F^2_\Phi = \tspan_\bC\{\x_1,\x_2\}$ where
\[
  \x_a \ = \ v_a \ + \ \sum_{j\ge 3} \x_a^j\,v_j \,,\quad a=1,2 \,.
\]
The $\x_a^j : B \cap X \to \bC$ are holomorphic and defined up to the action of the monodromy $\Gamma_X$.  The
\begin{eqnarray*}
  \x_3 & = & v_3 \,-\, \x_2^3\,v_4 \,-\, \x_1^3\,v_5 \\
  \x_r & = & v_r \,-\, \x_2^r\,v_4 \,-\, \x_1^r\,v_5 
  \,,\quad r \ge 6 \,.
\end{eqnarray*}
extend $\{\x_1,\x_2\}$ to a framing of $F^1(\Phi)$.  

By \cite[\eqref{extn-SE:A}]{Robles-extnHnorm}, the fibre $A$ is cut out by 
\begin{equation}\label{E:A2}
  A \ = \ 
  \{ \xi^3_1 \,,\, \xi^5_1 \,,\, \xi^r_1 \ = \ 0\} 
  \,\cap\, Z_1 \,.
\end{equation}
We have $\eta_0 = \x_1 \wedge \x_2$ and $\eta_\infty = \x_1\wedge \x_4$.

\subsection{Matrix coefficients in local coordinates}\label{S:matcoef2}

As discussed in \cite[\S\ref{extn-S:loc1}]{Robles-extnHnorm}, we have $\x_j = \exp(\ell(t_1)N) \z(t)\cdot(v_j)$, with $\z : U \to \exp(\fs_F^\perp)$ holomorphic.  We have $\fs_F^\perp = (\fs_F^\perp \cap \fz_N) \op \bC(v_3 \ot v^1 -  v_5 \ot v^3)$, with $\fz_N$ the centralizer of $N$ in $\fg_\bC$.  We may factor $\z = \exp \x^3_1(v_3 \ot v^1 -  v_5 \ot v^3) \cdot \a$ with $\a : U \to \exp(\fs_F^\perp \cap \fz_N)$ holomorphic.  Then
\begin{eqnarray*}
  \x_1 & = & \a_1 \,+\, 
  	\x^3_1\,\b_{1,2} \,+\,
	\ell(t_1)\,\x^3_1\,N \b_{1,2} \\
  \x_2 & = & \a_2 \,+\, \ell(t_1)\,N\,\a_2 \,+\,
  	\half\,\ell(t_1)^2\,N^2\,\a_2 \,,
  \qquad \x_4 \ = \ N^2\,\a_2 \ = \ v_4 \,,
\end{eqnarray*}
with $\b_{1,2} = v_3 - \half \x^3_1 v_5 : U \to W_2(V_\bC)$ holomorphic.  The condition \cite[\eqref{extn-E:locW}]{Robles-extnHnorm} is equivalent to the vanishing of $\x^3_1$ along $U \cap Z = \{t_1=0\}$.  In particular,
\begin{equation}\label{E:n31-2}
  \xi^3_1 \ = \ t_1\,\n^3_1
\end{equation}
for some holomorphic $\n^3_1 : U \to \bC$.  The fibre \eqref{E:A2} is locally characterized by 
\begin{equation} \label{E:A2loc}
  A \,\cap \, U \ = \ \{ t_1 ,\, \a^5_1 ,\, \a^r_1 = 0 \} \,.
\end{equation}

The following observations will simplify the computations that follow.  Set $\a_j = \a(v_j)$, and define $\a_j^i \in \cO(U)$ by $\a_j = \a_j^i\,v_i$.  The condition that $\a$ centralizes $N$ is equivalent to $\a^3_1 = 0$.  We also have $\x_2^3 = \a^3_2 - \bi\ell(t)$.  After a change of coordinates $t_1 \mapsto \exp(2\pi\a^3_2)\,t_1$, we also normalize $\a^3_2 = 0$.  To summarize $\a^3_1 ,\, \a^3_2 = 0$.  This implies that 
\begin{equation} \label{E:a3=0}
  0 \ = \ Q(\a_i , v_3) \ = \ Q(\a_i , N\a_j) \,,\qquad
  i,j=1,2\,.
\end{equation}

\subsection{The sections $\eta_0$ and $\eta_\infty$ in local coordinates}

Set
\begin{eqnarray*}
  \a_{0,6} & = & \a_1 \wedge \a_2 \,:\, 
	  U \ \to \ W_6(H_\bC) \\
  \b_{0,6} & = & \b_{1,2} \wedge \a_2 \,:\, 
  	U \ \to \ W_6(H_\bC) \\
  \b_{\infty,2} & = & \b_{1,2} \wedge (N^2\a_2) \,:\,
  	U \ \to \ W_2(H_\bC) \,.
\end{eqnarray*}
Then
\begin{eqnarray*}
  \eta_0 & = &  
  \a_{0,6} + \x^3_1\,\b_{0,6} 
  \,+\, \ell(t_1)\,N\,(\a_{0,6} + \x^3_1\,\b_{0,6}) 
  \,+\,
  	\half\,\ell(t_1)^2\,N^2\,(\a_{0,6} + \x^3_1\,\b_{0,6}) \\
  \eta_\infty & = & 
  N^2\,\a_{0,6} \ + \ \x^3_1\,\b_{\infty,2} \,.
\end{eqnarray*}

\subsection{The Hodge norms in local coordinates} \label{S:h2}

The extension of the Hodge norm on $Z \cap X$ to $X$ is $h = -\tRe\,Q(\eta_0 , \oleta_\infty)$.  There are two expressions for $h$ that will be useful.  Keeping \eqref{SE:Wn=2} in mind,  we have
\begin{eqnarray*}
  -Q(\eta_0 \,,\, \oleta_\infty) & = & 
  -Q( \x_1 \,,\, \overline{\x_1})\,
  	Q( \x_2 \,,\, \overline{\x_4})
  \ + \ 
  Q( \x_1 \,,\, \overline{\x_4})\,
  	Q( \x_2 \,,\, \overline{\x_1}) \\
  & = & -Q( \x_1 \,,\, \overline{\x_1}) 
  \ = \ 1 \,+\, |\x^3_1|^2 \,+\, |\x^5_1|^2 \,-\, 
  \sum_{r\ge6} |\x^r_1|^2 \,;
\end{eqnarray*}
with $\x^r_1 = \a^r_1$ and $\x^5_1 = \a^5_1 - \half(\x^3_1)^2$.
Alternatively, \eqref{E:a3=0} implies $\olx^3_1$ divides both $Q(\a_{0,6} , \olb_{\infty,2})$ and $Q(\a_{0,6} , N^2\olb_{0,6}) = Q(N^2\a_{0,6} , \olb_{0,6})$ in $\cC^\w(U)$, so that \eqref{E:n31-2} yields
\begin{eqnarray*}
  h & = & 
  - Q (\a_{0,6} + \x^3_1\,\b_{0,6} \,,\, 
  N^2 \ola_{0,6} + \olx^3_1\,\olb_{\infty,2}) \\
  & = & -Q(\a_{0,6} , N^2\ola_{0,6}) \,+\, 
  |\x^3_1|^2 (1 + \tfrac{1}{4}|\x^3_1|^2) 
  \,-\, \tRe\,\a^5_1 (\olx^3_1)^2 \,.
\end{eqnarray*}

The Hodge norm on $B \cap X$ is $h_0 = Q(\eta_0,\oleta_0)$.  Again, keeping \eqref{SE:Wn=2} in mind, in local coordinates we have
\begin{eqnarray*}
  h_0 & = & -\frac{(\log|t_1|^2)^2}{8\pi^2} \,
  Q \left( \a_{0,6} + \x^3_1\,\b_{0,6} \,,\, N^2(\ola_{0,6} 
  + \olx^3_1\,\olb_{0,6}) \right) \\ 
  & & + \ \bi \frac{\log |t_1|^2}{2\pi}
  Q \left( \a_{0,6} + \x^3_1\,\b_{0,6} \,,\, N(\ola_{0,6} 
  + \olx^3_1\,\olb_{0,6}) \right) \\
  & & + \ Q 
  \left( \a_{0,6} + \x^3_1\,\b_{0,6} \,,\, \ola_{0,6} 
  + \olx^3_1\,\olb_{0,6} \right) \\
  & = & 
  -\frac{(\log|t_1|^2)^2}{8\pi^2} \,
  \left\{
    Q ( \a_{0,6} \,,\, N^2\ola_{0,6}) \,+\, 
    |\x^3_1|^2\,(1 - \tfrac{1}{4} |\x^3_1|^2) \,+\,
    \tRe\,\a^5_1 (\olx^3_1)^2 
  \right\} \\
  & & - \ \frac{\log |t_1|^2}{2\pi}
  \left\{
  \olx^3_1\,Q(\a_1,\ola_2) + \x^3_1\,Q(\ola_1,\a_2)
  \, + \, O(|t_1|^2)
  \right\} \\
  & & + \ Q 
  \left( \a_{0,6} + \x^3_1\,\b_{0,6} \,,\, \ola_{0,6} 
  + \olx^3_1\,\olb_{0,6} \right)
  \,,
\end{eqnarray*}
with the second equality making use of \eqref{E:a3=0}.

\subsection{Plurisubharmonicity} \label{S:psh2}

Set $\rho_0 = 1/h_0$ and $\rho_1 = -\log h$. We know that $\rho_0$ is psh on $X$, and that $\rho_1$ is psh on $Z \cap X$.  The goal of this section is to show that $\rho_0+\rho_1$ is psh on $U$.  We follow the approach of \S\ref{S:psh1}.

\begin{lemma} 
Fix a holomorphic vector field $v$ on $U$ with $\tdv t_1 = 1$.  We have
\[
  \lim_{t_1\to0} \ddbv \, (\rho_0 + \rho_1) 
  \ = \ +\infty \,.
\]
\end{lemma}

\begin{proof}
Set 
\[
  q_2(\a) \ = \ -Q(\a_{0,6} , N^2\ola_{0,6}) 
  \ = \ -Q( \a_1 \,,\, \overline{\a_1}) 
  \ = \ 1 + |\a^5_1|^2 - \sum_{r\ge6} |\a^r_1|^2 \,.
\]
From \eqref{E:A2loc} we see that $q_2(\a)$ is identically one along $A \cap U$; shrinking $U$ if necessary, we may assume that $q_2(\a) > \half$ on $U$.  Then as $t_1 \to 0$, the expression $\ddbv\, (\rho_0 + \rho_1)$ is dominated by the term
$\displaystyle
  \frac{6\,q_2(\a)^2}{(8\pi^2)^2\,h_0^3} \,
  (\log |t_1|^2)^2\, \frac{\td t_1  \wedge\td \olt_1}{|t_1|^2}
$
in $\ddb \rho_0$.
\end{proof}

\begin{lemma} \label{L:psh2fnt}
Fix a holomorphic vector field $v$ on $U$ with $\tdv t_1 = 0$.  Shrinking $U$ if necessary, we have $\ddbv\, (\rho_0 + \rho_1) \ge 0$.
\end{lemma}

\begin{corollary} \label{C:phs2}
The function $\rho_0 + \rho_1$ is psh on $X$.
\end{corollary}

\subsection{Proof of Lemma \ref{L:psh2fnt}}

The proof of Lemma \ref{L:psh2fnt} is along the same lines as that of Lemma \ref{L:psh1fnt}, and we will be commensurately brief.

\subsubsection{Step 1}

We see from \eqref{E:n31-2} and \S\ref{S:h2} that the leading term, as $t_1\to 0$ in $\ddbv(\rho_0+\rho_1)$ is $-\log q_2(\a)$.  It follows from general principles, or directly from the first Hodge--Riemann bilinear relation $0 = Q(\a_1,\a_1) = 2\,\a^5_1 \,+\, \sum_{r\ge6} (\a_1^r)^2$, that $-\log q_2(\a)$ is psh.  If 
$\displaystyle \lim_{t_1\to0} -\ddbv\log q_2(\a)>0$, then
$\displaystyle \lim_{t_1\to0} \ddbv (\rho_0+\rho_1)>0$.

\subsubsection{Step 2}

Suppose 
$\displaystyle \lim_{t_1\to0} -\ddbv\log q_2(\a)=0$.  This is equivalent to 
$\displaystyle \lim_{t_1\to0} \tdv \a_1^r ,\, \tdv\a_1^5=0$, and we have 
$\displaystyle \lim_{t_1\to0} \ddbv (\rho_0+\rho_1)=0$. 
Set
\[
  q_0(\a) \ = \ Q(\a_{0,6},\ola_{0,6}) 
  \ =  \ Q(\a_1,\ola_1) Q(\a_2,\ola_2) \,-\,
  \left| Q(\a_1,\ola_2) \right|^2 \,.
\]
Then
\begin{eqnarray*}
  \dv \, q_0(\a) & = & 
  Q(\a_1,\ola_1)\, 
  \left( \tdv\a^4_2 - \ola^5_2\,\tdv\a^5_2 +
  	\tsum \ola^r_2\,\tdv \a^r_2 \right) \\
  & & - \ 
  Q(\ola_1,\a_2)\,\tdv  \a^4_1
  \ + \ Q(\a_1,\ola_2)\,
  \left(
    \ola^5_1\,\tdv \a^5_2 \,-\, \tsum \ola^r_1\,\tdv \a^r_2
  \right) \\
  \ddbv\,q_0(\a) & = & 
  - Q(\a_1,\ola_1)\, 
  \left(\left|\tdv\a^5_2\right|^2 \,-\,
  	\tsum \left|\tdv \a^r_2\right|^2 \right)
	\ - \ \left| \tdv\a^4_1 \right|^2  \\
  & & - \ 
  \left|
    \ola^5_1 \tdv\a^5_2 \,-\, \tsum \ola^r_1 \tdv\a^r_2
  \right|^2 \,.
\end{eqnarray*}
The first Hodge--Riemann bilinear relation yields $0 = Q(\a_2,\a_2) = 2\a^4_2 \,+\, \sum (\a^r_2)^2$.   And the infinitesimal period relation $0 = Q(\x_1,\td \x_2) = \td\x^5_2 \,+\, \x^3_1\,\td\x^3_2 \,+\, \x^r_1\,\td\x^r_2$ implies $0 = \tdv \a^5_2 + \a^r_1\,\tdv \a^r_2$,
so that $\tdv\a^5_2 = 0$ along $A \cap  U$.  Then  implies that $q_0(\a) < 0$ in a neighborhood of $A \cap U$.  If $\lim_{t_1\to0} |\tdv \a^4_1|^2 + |\tdv \a^r_2|^2 > 0$, then $\ddbv(\rho_0+\rho_1)$ is asymptotically dominated by the term $-(1/h_0^2) \ddbv\,q_0(\a)$ as $t_1 \to 0$, and we have $\ddbv (\rho_0 + \rho_1) > 0$ in a punctured neighborhood $\{ 0 < |t_1| < \e\}$ of $Z_1 \cap U = \{t_1=0\}$. 

\subsubsection{Step 3}

Now assume that 
$\displaystyle \lim_{t_1\to0} \tdv\a_1,\,\tdv\a_2=0$.  Differentiating the infinitesimal period relation yields 
\begin{eqnarray*}
  0 & = & Q(\td\x_1,\td\x_2) 
  \ = \ 
  \td \x^3_1\wedge \td \x^3_2 
  	\,+\, \tsum_r\,\td \x^r_1 \wedge \td \x^r_2 \\
  & = & \td (t_1\,\n^3_1) \wedge \td(-\bi\ell(t))
  	\,+\, \tsum_r\,\td \a^r_1 \wedge \td \a^r_2
  \ = \ \tfrac{1}{2\pi} \td t_1 \wedge \td\n^3_1
  \,+\,\tsum_r\,\td \a^r_1 \wedge \td \a^r_2 \,.
\end{eqnarray*}
Evaluating this 2-form on $\partial_{t_1} \wedge v$ yields $\td_v\n^3_1 = 2\pi \sum_r (\tdv\a^r_1 \cdot \td_{\partial_{t_1}}\a^r_2 - \td_{\partial_{t_1}}\a_1^r \cdot \tdv\a_2^r)$.  In particular, $\displaystyle \lim_{t_1\to0} \tdv \n^3_1 = 0$, and the order of vanishing of $\tdv\n^3_1$ as $t_1\to0$ is bounded below by the order of vanishing of $\tdv\a_1, \tdv\a_2$.   It follows from the expressions for $h$ and $h_0$ in \S\ref{S:h2} that we have $\ddbv (\rho_0 + \rho_1) \ge 0$ in a neighborhood of $Z_1 \cap U = \{t_1=0\}$. 

This completes the proof of Lemma \ref{L:psh2fnt}. \hfill\qed

\subsection{Pseudoconvexity} \label{S:pcnvx2}

We now prove Theorem \ref{T:pseudocnvx} for degenerations of type \eqref{E:hd2}, following the argument outlined in \S\ref{S:prf-pcnvx}.  The analysis in the proof of Lemma \ref{L:psh2fnt} implies \S\ref{S:prf-pcnvx}\ref{i:psh}.  It remains to observe that $\rho$ satisfies \S\ref{S:prf-pcnvx}\ref{i:A}. As noted in that discussion, we have $\rho_0 \ge 0$ with $\{ \rho_0 = 0 \} =  Z \cap X$.  The fibre $A$ is locally characterized by 
\[
  A \,\cap\, U \ = \ 
  \{ t_1 , \x^3_1 , \x^5_1 , \x^r_1  = 0 \} \,.
\]
So along $Z \cap X$ we have $\rho_1 \ge 0$ with $A = \{ \rho_1 = 0\}$.  Analysis similar to that in in \S\ref{S:psh2} implies that, shrinking $X$ if necessary, we have $\rho \ge 0$ on $X$ and $A = \{ \rho = 0 \}$.  

\section{The third degeneration} \label{S:deg3}

The next type of degeneration has limiting mixed Hodge structure $(W,F,N)$ satisfying $W_0(V)=0$ and $W_3(V) = W_4(V) = V$, and with Hodge numbers $\bfh_1 = (2,2)$ and $\bfh_2 = (0,\sfh-4,0)$.  The associated Hodge diamond $\Diamond(V)$ is
\begin{equation}\label{E:hd3}
\begin{tikzpicture}[baseline={([yshift=-.5ex]current bounding box.center)},scale=1.1]
  \draw [<->] (0,2.75) -- (0,0) -- (2.75,0);		
  \draw [gray] (1,0) -- (1,2);
  \draw [gray] (2,0) -- (2,2);
  \draw [gray] (0,1) -- (2,1);
  \draw [gray] (0,2) -- (2,2);
  \draw [fill] (0,1) circle [radius=0.06];
  \node [above right] at (0,1) {{$v_7$}};
  \node [below right] at (0,1) {{$v_8$}};
  \draw [fill] (1,2) circle [radius=0.06];
  \node [above right] at (1,2) {{$v_3$}};
  \node [below right] at (1,2) {{$v_4$}};
  \draw [fill] (1,1) circle [radius=0.06];
  \node [above right] at (1,1) {$v_r$};
  \draw [fill] (1,0) circle [radius=0.06];
  \node [above right] at (1,0) {{$v_5$}};
  \node [below right] at (1,0) {{$v_6$}};
  \draw [fill] (2,1) circle [radius=0.06];
  \node [above right] at (2,1) {{$v_1$}};
  \node [below right] at (2,1) {{$v_2$}};
\end{tikzpicture}
\end{equation}
See \eqref{E:hd2-3} for the diamond $\Diamond(H)$.

We may choose a basis of $\{v_1,\ldots,v_\sfd\}$ of $V_\bC$ so that:  the polarization satisfies $Q(v_a,v_b) = \d^9_{a+b}$ for all $1 \le a,b\le 8$, $Q(v_r,v_s) = \d_{rs}$ for all $9 \le r,s$, and all other pairings are zero; the underlying real structure is $\overline{v_1} = v_3$, $\overline{v_2} = v_4$, $\overline{v_5} = v_7$, $\overline{v_6} = v_8$ and $\overline{v_r} = v_r$; the nilpotent operator is given by
\begin{equation} \nonumber
  N \ = \ \bi(v_6 \ot v^1 - v_8 \ot v^3) 
  \,+\, \bi(v_5 \ot v^2 - v_7 \ot v^4) \,;
\end{equation}
the Hodge filtration is $F^2(V_\bC) = \tspan_\bC\{v_1,v_2\}$, and the weight filtration is $W_1(V_\bC) = \tspan_\bC\{ v_5,\ldots,v_8 \}$.  We have $e_0 = v_1 \wedge v_2$ and $e_\infty = v_5\wedge v_6$, in the notation of \cite[\S\S\ref{extn-S:pmr}-\ref{extn-S:einfty}]{Robles-extnHnorm}.

\subsection{Period matrix representation}

There exist holomorphic functions $\x_1^j,\x_2^j : B \cap X \to \bC$, defined up to the action of $\Gamma_X$, so that $\x_1 = v_1 + \sum_{j\ge 3} \x_1^j\,v_j$, $\x_2 = v_2  + \sum_{j\ge3}\x_2^j$ frames $F^2(\Phi)$, and $\x_3 = v_3 - \x^6_2\,v_7 - \x^6_1\,v_8$, $\x_4 = v_4 - \x^5_2\,v_7 - \x^5_1\,v_8$, $\x_5 = v_5 - \x^4_2\,v_7 - \x^4_1\,v_8$, $\x_6 = v_6 - \x^3_2\,v_7 - \x^3_1\,v_8$, $\x_r = v_r - \x^r_2\,v_7 - \x^r_1\,v_8$ frames $F^1(\Phi)$. 

By \cite[\eqref{extn-SE:A}]{Robles-extnHnorm}, the fibre $A$ is cut out by 
\begin{equation}\label{E:A3}
  A \ = \ \{ \xi^3_1 \,,\, \xi^4_1 \,,\, 
  	\xi^3_2 \,,\, \xi^4_2 \ = \ 0\} \,\cap\, Z_1 \,.
\end{equation}
We have $\eta_0 = \xi_1 \wedge \xi_2$ and $\eta_\infty = \xi_5 \wedge \xi_6$.  

\subsection{Matrix coefficients in local coordinates}\label{S:matcoef3}

As discussed in \cite[\S\ref{extn-S:loc1}]{Robles-extnHnorm}, we have $\x_j = \exp(\ell(t_1)N) \z(t)\cdot(v_j)$, with $\z : U \to \exp(\fs_F^\perp)$ holomorphic.  We have $\fs_F^\perp = (\fs_F^\perp \cap \fz_N) \op \bC\,b$, with $\fz_N$ the centralizer of $N$ in $\fg_\bC$ and $b = (v_3 \ot v^2 -  v_7 \ot v^6 -v_4 \ot v^1 + v_8 \ot v^5)$.  Keeping in mind the condition \cite[\S\ref{extn-S:Aloc}]{Robles-extnHnorm} that the restriction of $\z$ to $Z_1 \cap U$ centralizes $N$, we may factor $\z = \exp (t_1 \n\,b) \cdot \a$ with $\a : U \to \exp(\fs_F^\perp \cap \fz_N)$ and $\n : U \to  \bC$ both holomorphic.  

Define $\b_i = b(\a_i)$.  Then 
\begin{eqnarray*}
  \x_i & = & \a_i + t_1\n\,\b_i \,+\, 
  \ell(t_1) N(\a_i + t_1\n\,\b_i) \,,\qquad i=1,2 \,,\\
  \x_j & = & \a_j + t_1\n\,\b_j \,,\qquad j=5,6 \,, \\
  \a_5 & = & -\bi N\a_2  \tand
  \a_6 \ = \ -\bi N\a_1 \,.
\end{eqnarray*}

Set $\a_j = \a(v_j)$, and define $\a_j^i \in \cO(U)$ by $\a_j = \a_j^i\,v_i$; the condition that $\a$ centralize $N$ implies 
\begin{equation}\label{E:kerN-3}
  \a^4_1 \ = \ \a^3_2 \,.
\end{equation}
The fibre \eqref{E:A3} is locally characterized by 
\begin{equation} \label{E:A3loc}
  A \,\cap \, U \ = \ \{ t_1 ,\, \a^3_1 \,,\, 
  	\a^4_1 \,,\, \a^3_2 \,,\, \a^4_2 \ = \ 0 \} \,.
\end{equation}

\subsection{The sections $\eta_0$ and $\eta_\infty$ in local coordinates}

Define $\a_0 = \a_1 \wedge \a_2$, $\b_0 = \a_1 \wedge \b_2 + \b_1 \wedge \a_2 + t_1 \n \b_1 \wedge\b_2)$  and $\b_\infty = -\bi (N\a_2) \wedge \b_6 - \b_5 \wedge \bi(N\a_1)
  + t_1\n \b_5 \wedge \b_6$.  Then 
\begin{eqnarray*}
  \eta_0 & = & \a_0 + t_1 \n \b_0 
  \,+\, \ell(t_1) N (\a_0 + t_1 \n \b_0)
  \,+\, \half \ell(t_1)^2 N^2 (\a_0 + t_1 \n \b_0) \\
  \eta_\infty & = & 
  \half N^2 \a_0 \ + \ 
  t_1\n \b_\infty \,.
\end{eqnarray*}

\subsection{The Hodge norms in local coordinates} \label{S:h3}

The extension of the Hodge norm on $Z \cap X$ to $X$ is $h = -\tRe\,Q(\eta_0 , \oleta_\infty)$.  Keeping \eqref{SE:Wn=2} in mind, we have
\begin{eqnarray*}
  h & = &
  -\half Q(\a_0 , N^2 \ola_0 ) \\
  & & + \ 
  \tRe \left\{
    t_1\n \left[ Q(\ola_0,\b_\infty) + Q(\b_\infty,\half N^2\ola_0) \right]
    \,+\, |t_1\n|^2\,Q(\b_0,\olb_\infty)
  \right\}
\end{eqnarray*}

The Hodge norm on $B \cap X$ is $h_0 = Q(\eta_0,\oleta_0)$.  Again, keeping \eqref{SE:Wn=2} in mind, in local coordinates we have
\begin{eqnarray*}
  h_0 & = &
  -\frac{(\log|t_1|^2)^2}{8\pi^2} \,
  Q \left( \a_0 + t_1 \n \b_0 \,,\, N^2(\overline{a_0 + t_1 \n \b_0}) \right) \\ 
  & & + \ \bi \frac{\log |t_1|^2}{2\pi}
  Q \left( \a_0 + t_1 \n \b_0 \,,\, N(\overline{\a_0 + t_1 \n \b_0}) \right) \\
  & & + \ Q 
  \left( \a_0 + t_1 \n \b_0 \,,\, \overline{\a_0 + t_1 \n \b_0} \right) \,.
\end{eqnarray*}

\subsection{Plurisubharmonicity} \label{S:psh3}

Set $\rho_0 = 1/h_0$ and $\rho_1 = -\log h$. We know that $\rho_0$ is psh on $X$, and that $\rho_1$ is psh on $Z \cap X$.  The goal of this section is to show that $\rho_0+\rho_1$ is psh on $U$.  We follow the approach of \S\ref{S:psh1} and \S\ref{S:psh2}.

\begin{lemma} 
Fix a holomorphic vector field $v$ on $U$ with $\tdv t_1 = 1$.  We have
\[
  \lim_{t_1\to0} \ddbv \, (\rho_0 + \rho_1) 
  \ = \ +\infty \,.
\]
\end{lemma}

\begin{proof}
From \eqref{E:kerN-3} and \eqref{E:A3loc} we see that 
\begin{eqnarray*}
  q_2(\a) & = & -Q(\a_0 , N^2\ola_0) \\
  & = & 
  (1 - |\a^3_1|^2 - |\a^4_1|^2) (1 - |\a^3_2|^2 - |\a^4_2|^2)
  \ - \ |\a^3_1\,\ola^3_2 + \a^4_1\,\ola^4_2|^2
\end{eqnarray*}
is identically one along $A \cap U$; shrinking $U$ if necessary, we may assume that $q_2(\a) > \half$ on $U$.  Then as $t_1 \to 0$, the expression $\ddbv\, (\rho_0 + \rho_1)$ is dominated by the term
$\displaystyle
  \frac{6\,q_2(\a)^2}{(8\pi^2)^2\,h_0^3} \,
  (\log |t_1|^2)^2\, \frac{\td t_1  \wedge\td \olt_1}{|t_1|^2}
$
in $\ddb \rho_0$.
\end{proof}

\begin{lemma} \label{L:psh3fnt}
Fix a holomorphic vector field $v$ on $U$ with $\tdv t_1 = 0$.  Shrinking $U$ if necessary, we have $\ddbv\, (\rho_0 + \rho_1) \ge 0$.
\end{lemma}

\begin{corollary} \label{C:phs3}
The function $\rho_0 + \rho_1$ is psh on $X$.
\end{corollary}

\subsection{Proof of Lemma \ref{L:psh3fnt}}

The proof is along the same lines as that of Lemmas \ref{L:psh1fnt} and \ref{L:psh2fnt}, and we will be commensurately brief.

\subsubsection{Step 1}

We see from \S\ref{S:h3} that the leading term, as $t_1\to 0$ in $\ddbv(\rho_0+\rho_1)$ is $-\log q_2(\a)$.  It follows from general principles (or by direct computation, with the first Hodge--Riemann bilinear relation $0=Q(\a_i,\a_j)$, $i,j=1,2$) that $-\log q_2(\a)$ is psh.  If 
$\displaystyle \lim_{t_1\to0} -\ddbv\log q_2(\a)>0$, then
$\displaystyle \lim_{t_1\to0} \ddbv (\rho_0+\rho_1)>0$.

\subsubsection{Step 2}

Suppose 
$\displaystyle \lim_{t_1\to0} -\ddbv\log q_2(\a)=0$.  Since $-\ddb \log q_2(\a)=0$ if and only if $\tdv \a_1^3 ,\, \tdv \a_1^4 ,\, \tdv \a_2^3 ,\, \tdv \a_2^4 = 0$, this is equivalent to 
\begin{subequations}\label{SE:W3(a)=0}
\begin{equation}
  \lim_{t_1\to0} \tdv \a_1^3 ,\, \tdv \a_1^4 ,\, \tdv \a_2^3 ,\, \tdv \a_2^4 
  \ = \ 0 \,,
\end{equation}
which may be rephrased as 
\begin{equation}
  \lim_{t_1\to0} \tdv \a_1 , \tdv\a_2 \equiv 0
  \quad \hbox{modulo} \quad W_2(V_\bC) \,. 
\end{equation}
\end{subequations}
In this situation we have 
$\displaystyle \lim_{t_1\to0} \ddbv (\rho_0+\rho_1)=0$.  
Set
\begin{eqnarray*}
  q_1(\a) & = & \bi Q(\a_0,N\ola_0) 
  \ = \ 
  \bi Q \left(\a_1\wedge\a_2 \,,\, 
  	(N\ola_1) \wedge \ola_2 + \ola_1 \wedge (N\ola_2) \right) \\
  & = & 
  \bi Q(\a_1 , N\ola_1) Q(\a_2 , \ola_2 )
  \,-\, \bi Q(\a_2 , N\ola_1) Q(\a_1 , \ola_2 ) \\
  & &
  + \ \bi Q(\a_2 , N\ola_2) Q(\a_1 , \ola_1 )
  \,-\, \bi Q(\a_1 , N\ola_2) Q(\a_2 , \ola_1 ) \,.
\end{eqnarray*}
If
\begin{equation}\label{E:deg3s2}
  \lim_{t_1\to0} \ddbv \, q_1(\a) \ \not= \ 0 \,,
\end{equation}
then \S\ref{S:h3} implies that the term 
$\displaystyle -\frac{(\log|t_1|^2)^3}{16 \pi^3\,h_0^3} q_2(\a) \ddbv\,q_1(\a)$ in $-(\ddbv\,h_0)/h_0^2$ dominates $\ddbv (\rho_0+\rho_1)$ as $t_1\to \infty$.
It follows from \eqref{SE:Wn=2} and \eqref{SE:W3(a)=0} that
\begin{eqnarray*}
  \lim_{t_1\to0} \dbv \, q_1(\a) 
  & = & 
  \bi Q(\a_1 , N\ola_1) Q(\a_2 , \overline{\tdv \a_2} )
  \,-\, \bi Q(\a_2 , N\ola_1) Q(\a_1 , \overline{\tdv \a_2} ) \\
  & &
  + \ \bi Q(\a_2 , N\ola_2) Q(\a_1 , \overline{\tdv \a_1} )
  \,-\, \bi Q(\a_1 , N\ola_2) Q(\a_2 , \overline{\tdv \a_1} ) \,,\\
  \lim_{t_1\to0} \ddbv \, q_1(\a) 
  & = & 
  \bi Q(\a_1 , N\ola_1) Q(\tdv \a_2 , \overline{\tdv \a_2} )
  \,-\, \bi Q(\a_2 , N\ola_1) Q(\tdv \a_1 , \overline{\tdv \a_2} ) \\
  & &
  + \ \bi Q(\a_2 , N\ola_2) Q(\tdv \a_1 , \overline{\tdv \a_1} )
  \,-\, \bi Q(\a_1 , N\ola_2) Q(\tdv \a_2 , \overline{\tdv \a_1} ) \\
  & = & 
  \left(1 - |\a_1^3|^2 - |\a_1^4|^2\right) \, \tsum_{r\ge9} |\tdv \a_2^r|^2 \\
  & & + \ 
  \left(1 - |\a_2^3|^2 - |\a_2^4|^2\right) \, \tsum_{r\ge9} |\tdv \a_1^r|^2 \\
  &  & + \ 2\tRe(\a^4_1\ola^4_2) \,
  	\tsum_{r\ge9} \tdv \a_2^r \,\overline{\tdv \a_1^r}
  \ + \ 2\tRe(\a^3_1\ola^3_2) \,
  	\tsum_{r\ge9} \tdv \a_1^r \,\overline{\tdv \a_2^r}\,.
\end{eqnarray*}
It follows from \eqref{E:A3loc} that, after shrinking $U$ if necessary, \eqref{E:deg3s2} holds if and only if 
$\displaystyle \lim_{t_1\to0} \sum_{r\ge9} (|\tdv \a_1^r|^2 + |\tdv \a_2^r|^2) > 0$.  And in this case we have $\ddbv (\rho_0 + \rho_1) > 0$ in a punctured neighborhood $\{ 0 < |t_1| < \e\}$ of $Z_1 \cap U = \{t_1=0\}$. 

\subsubsection{Step 3}

Now suppose that both \eqref{SE:W3(a)=0} and 
$\displaystyle \lim_{t_1\to0} \sum_{r\ge9} (|\tdv \a_1^r|^2 + |\tdv \a_2^r|^2) = 0$ hold.   Together these are equivalent to 
\begin{equation}\label{E:W2(a)=0}
  \lim_{t_1\to0} \tdv \a_1 , \tdv\a_2 \equiv 0
  \quad \hbox{modulo} \quad W_1(V_\bC) \,. 
\end{equation}
Set
\[
  q_0(\a) \ = \ Q(\a_0,\ola_0) 
  \ =  \ Q(\a_1,\ola_1) Q(\a_2,\ola_2) \,-\,
  \left| Q(\a_1,\ola_2) \right|^2 \,,
\]
and $\w = -q_2(\a) \,\ddb q_0(\a) + 4 \, \partial q_1(\a) \wedge \dbar q_1(\a)$.  If
$\displaystyle
  \lim_{t_1\to0} \w(v,\olv) \not= 0 
$,
then \S\ref{S:h3} implies that the term 
$\displaystyle \frac{(\log|t_1|^2)^2}{8 \pi^2\,h_0^3} \w(v,\olv)$ in $\ddbv\,\rho_0 = \ddbv\,(1/h_0)$ dominates $\ddbv (\rho_0+\rho_1)$ as $t_1\to \infty$.

We claim that 
$\displaystyle
  \lim_{t_1\to0} \w(v,\olv) > 0
$
if and only if 
\begin{equation}\label{E:deg3s3}
  \lim_{t_1\to0} \,
  \left| \tdv \a_1^5 \right|^2 \,+\, \left| \tdv \a_1^6 \right|^2 \,+\,
  \left| \tdv \a_2^5 \right|^2 \,+\, \left| \tdv \a_2^6 \right|^2 \ > \ 0\,.
\end{equation}
To see this, note that \eqref{SE:Wn=2} and \eqref{E:W2(a)=0} yields
\begin{eqnarray*}
  \lim_{t_1\to0} \dbv \, q_1(\a) 
  & = & 
  \left( 1 - |\a_1^3|^2 - |\a_1^4|^2 \right)
  	\left( \overline{\tdv \a_2^5} \,+\, \a_2^4\,\overline{\tdv \a_2^7} 
		\,+\, \a_2^3\,\overline{\tdv \a_2^8} \right) \\
  & & + \ 
  \left( 1 - |\a_2^3|^2 - |\a_2^4|^2 \right)
  \left( \overline{\tdv \a_1^6} \,+\, \a_1^4\,\overline{\tdv \a_1^7} 
  	\,+\, \a_1^3\,\overline{\tdv \a_1^8} \right) \\
  & & + \ 
  2\,\tRe(\a^3_1\ola^3_2) \left(
    \overline{\tdv \a_2^6} \,+\, \a_1^4\,\overline{\tdv \a_2^7} 
    \,+\, \a_1^3\,\overline{\tdv \a_2^8}
  \right) \\
  & & + \ 
  2\,\tRe(\a^4_1\ola^4_2) \left(
    \overline{\tdv \a_1^5} \,+\, \a_2^4\,\overline{\tdv \a_1^7}
    \,+\, \a_2^3\,\overline{\tdv \a_1^8}
  \right) \,.
\end{eqnarray*}
Likewise
\begin{eqnarray*}
  \lim_{t_1\to0} \ddbv \, q_0(\a) 
  & = & 
  2\,\tRe\left\{
  	(\tdv \a_2^5 \,+\, \ola^4_2\,\tdv \a^7_2 \,+\, \ola^3_2\,\tdv \a^8_2)
	(\tdv \ola_1^6 \,+\, \a^4_1\,\tdv \ola^7_1 \,+\, \a^3_1\,\tdv \ola^8_1)
  \right\} \\
  & & -
  \left|
    \tdv \a_1^5 \,+\, \ola^4_2\,\tdv \a^7_1 \,+\, \ola^3_2\,\tdv \a^8_1
  \right|^2
  - 
  \left|
    \tdv \a_2^6 \,+\, \ola^4_1\,\tdv \a^7_2 \,+\, \ola^3_1\,\tdv \a^8_2
  \right|^2 \,. 
\end{eqnarray*}
Keeping \eqref{E:A3loc} in mind, and shrinking $U$ if necessary, the claim follows from these limits; the first Hodge--Riemann bilinear relation $0 = Q(\a_i,\a_j)$, for $i,j=1,2$, which reads
\begin{eqnarray*}
   -\a_1^8 & = & 
   \a_1^3 \,\a_1^6 \,+\, \a_1^4 \,\a_1^5 \,+\, \half \tsum_r\, (\a_1^r)^2 \\
   -\a_2^7
   & = & 
   \a_2^3 \,\a_2^6 \,+\, \a_2^4 \,\a_2^5 \,+\, \half \tsum_r\, (\a_2^r)^2 \\
   -\a_1^7 \,-\, \a_2^8
   & = & 
   \a_1^3 \,\a_2^6 \,+\, \a_1^4 \,\a_2^5 \,+\, \a_1^5 \,\a_2^4 \,+\, 
   \a_1^6 \,\a_2^3 \,+\, \tsum_r\, \a_1^r \,\a_2^r \,;
\end{eqnarray*}
and the infinitesimal period relation $0 = Q(\x_1,\td\x_2) = Q(\x_2,\td\x_1)$, which yields
\begin{eqnarray*}
  -\lim_{t_1\to0} \td \a^7_1 & = & 
  	\a_2^4\,\tdv\a^5_1 \,+\, \a_2^3\,\tdv\a^6_1 \\
  -\lim_{t_1\to0} \td \a^8_2 & = & 
  	\a_1^4\,\tdv\a^5_2 \,+\, \a_1^3\,\tdv\a^6_2 \,.
\end{eqnarray*}

\subsubsection{Step 4}

It remains to consider the case that
$\displaystyle \lim_{t_1\to0} \tdv \a_1 , \tdv \a_2 = 0$.  Differentiating the infinitesimal period relation yields 
\begin{eqnarray*}
  0 & = & Q(\td\x_1 , \td \x_2) \\
  & = &
  \td\x^3_1 \wedge \td\x^6_2 \,+\, \td\x^4_1 \wedge \td\x^5_2 
  	\,+\, \td\x^5_1 \wedge \td\x^4_2 \,+\, \td\x^6_1 \wedge \td\x^3_2 
	\,+\, \tsum  \td\x^r_1 \wedge \td\x^r_2 \\
  & = &
  \tfrac{1}{\pi}\td t_1 \wedge \td \n \ + \ 
  (\td \a^5_2 + \td \a^6_1) \wedge \td (t_1\n) \\
  & & + \ 
  \td\a^3_1 \wedge \td\a^6_2 \,+\, \td\a^4_1 \wedge \td\a^5_2 
  	\,+\, \td\a^5_1 \wedge \td\a^4_2 \,+\, \td\a^6_1 \wedge \td\a^3_2 
	\,+\, \tsum  \td\a^r_1 \wedge \td\a^r_2 \,.
\end{eqnarray*}
Evaluating this 2-form on $\partial_{t_1} \wedge v$ implies 
$\displaystyle \lim_{t_1\to0} \tdv \n =  0$, and that the order of vanishing of $\tdv\n$ as $t_1\to0$ is bounded below by the order of vanishing of $\tdv\a_1, \tdv\a_2$.   It follows from the expressions for $h$ and $h_0$ in \S\ref{S:h3} that we have $\ddbv (\rho_0 + \rho_1) \ge 0$ in a neighborhood of $Z_1 \cap U = \{t_1=0\}$. 

This completes the proof of Lemma \ref{L:psh3fnt}. \hfill\qed

\subsection{Pseudoconvexity} \label{S:pcnvx3}

We now prove Theorem \ref{T:pseudocnvx} for degenerations of type \eqref{E:hd3}, following the argument outlined in \S\ref{S:prf-pcnvx}.  The analysis in the proof of Lemma \ref{L:psh3fnt} implies \S\ref{S:prf-pcnvx}\ref{i:psh}.  It remains to observe that $\rho$ satisfies \S\ref{S:prf-pcnvx}\ref{i:A}. As noted in that discussion, we have $\rho_0 \ge 0$ with $\{ \rho_0 = 0 \} =  Z \cap X$.  The fibre $A$ is locally characterized by 
\[
  A \,\cap\, U \ = \ 
  \{ t_1 , \x^3_1 , \x^4_1 , \x^3_2 , \x^4_2 = 0 \} \,.
\]
So along $Z \cap X$ we have $\rho_1 \ge 0$ with $A = \{ \rho_1 = 0\}$.  Analysis similar to that in in \S\ref{S:psh3} implies that, shrinking $X$ if necessary, we have $\rho \ge 0$ on $X$ and $A = \{ \rho = 0 \}$.  

\section{The last two degenerations} \label{S:lastdeg}

\subsection{The penultimate degeneration}

The fourth type of degeneration has limiting mixed Hodge structure $(W,F,N)$ with Hodge numbers $\bfh_0 = (1)$, $\bfh_1 = (1,1)$ and $\bfh_2 = (0,\sfh-2,0)$.  The associated Hodge diamond $\Diamond(V)$ is
\begin{equation}\label{E:hd4}
\begin{tikzpicture}[baseline={([yshift=-.5ex]current bounding box.center)},scale=1.1]
  \draw [<->] (0,2.75) -- (0,0) -- (2.75,0);		
  \draw [gray] (1,0) -- (1,2);
  \draw [gray] (2,0) -- (2,2);
  \draw [gray] (0,1) -- (2,1);
  \draw [gray] (0,2) -- (2,2);
  \draw [fill] (0,1) circle [radius=0.06];
  \node [above right] at (0,1) {{$v_7$}};
  \draw [fill] (0,0) circle [radius=0.06];
  \node [above right] at (0,0) {{$v_6$}};
  \draw [fill] (1,2) circle [radius=0.06];
  \node [above right] at (1,2) {{$v_3$}};
  \draw [fill] (1,1) circle [radius=0.06];
  \node [above right] at (1,1) {$v_4$};
  \node [below right] at (1,1) {$v_r$};
  \draw [fill] (1,0) circle [radius=0.06];
  \node [above right] at (1,0) {{$v_5$}};
  \draw [fill] (2,2) circle [radius=0.06];
  \node [above right] at (2,2) {{$v_2$}};
  \draw [fill] (2,1) circle [radius=0.06];
  \node [above right] at (2,1) {{$v_1$}};
\end{tikzpicture}
\end{equation}
See \eqref{E:hd2-4} for the Hodge diamond $\Diamond(H)$.

We may choose a basis of $\{v_1,\ldots,v_\sfd\}$ of $V_\bC$ so that:  the polarization satisfies $Q(v_a,v_b) = \d^8_{a+b}$ for all $1 \le a,b\le 7$, $Q(v_r,v_s) = \d_{rs}$ for all $8 \le r,s$, and all other pairings are zero; the underlying real structure is $\overline{v_1} = v_3$, $\overline{v_2} = v_2$, $\overline{v_4} = -v_4$, $\overline{v_5} = v_7$, $\overline{v_6} = v_6$ and $\overline{v_r} = v_r$; the nilpotent operator is given by
\begin{equation} \nonumber
  N \ = \ \bi(v_5 \ot v^1 - v_7 \ot v^3) 
  \,+\, \bi( v_6\ot v^4 - v_4 \ot v^2 ) \,;
\end{equation}
the Hodge filtration is $F^2(V_\bC) = \tspan_\bC\{v_1,v_2\}$, and the weight filtration is $W_0(V_\bR) = \tspan_\bR\{ v_6 \}$ and $W_1(V_\bC) = \tspan_\bC\{ v_5 , v_6 , v_7 \}$.  We have $e_0 = v_1 \wedge v_2$ and $e_\infty = v_5\wedge v_6$, in the notation of \cite[\S\S\ref{extn-S:pmr}-\ref{extn-S:einfty}]{Robles-extnHnorm}.

The filtration $F^2(\Phi)$ is framed by $\x_1 = v_1 + \sum_{j\ge3} \x^j_1\,v_1$ and $\x_2 = v_2 + \sum_{j\ge3} \x^j_3\,v_j$.  The collection $\{\x_1,\x_2\} \cup \{ \x_a = v_a + \x_a^6\,v_6 + \x_a^7\,v_7 \ | \ a=3,4,5\,,\  a \ge 8\}$ frames $F^1(\Phi)$.  Here the $\x^j_i$ a holomorphic  functions on $B \cap X$ are defined up to the action of $\Gamma_X$.  

We have
\[
  \eta_0 \ = \ \x_1 \wedge \x_2 \tand
  \eta_\infty \ = \ \x_5 \wedge \x_6 
\]
and
\begin{eqnarray*}
  h \  = \ Q(\eta_0 , \oleta_\infty) & = & 1 \,-\, |\x_1^3|^2
\end{eqnarray*}
and $-\log h$ is psh on $X$.  Set $\rho_1 = -\log h$, and let $\rho_0$ be the smooth function given by \eqref{E:rho0a}.  By \cite[\eqref{extn-SE:A}]{Robles-extnHnorm}, the fibre $A \subset Z$ is cut out by $\{ \x^3_1=0 \}$.  It is easily verified that the conditions \S\ref{S:prf-pcnvx}\ref{i:rho0}-\ref{i:A} hold, establishing Theorem \ref{T:pseudocnvx} for degenerations of type \eqref{E:hd4}.   In this case the psh exhaustion $\rho$ is smooth.

\subsection{The Hodge--Tate degeneration} \label{S:ht}

These limiting mixed Hodge structures are characterized by $W_0(V) = W_1(V)$ and $W_2(V) = W_3(V)$; and with Hodge numbers $\bh_0 =2$ and $\bfh_2 = (0,\sfh,0)$.  The associated Hodge diamonds are given in \eqref{E:hd2-5}.

The proof of Theorem \ref{T:pseudocnvx} is trivial for Hodge--Tate degenerations: The function $h$ is identically 1, so that $\rho_1 = 0$.  We must have $A = Z_1^* =  Z_1$, and the smooth function $\rho_0$ defined in \eqref{E:rho0a} yields the desired psh exhaustion of $X$.

\appendix

\section{Hodge diamonds} \label{S:hd}

Given a mixed Hodge structure $(W,F)$ on a vector space $V$ the \emph{Hodge diamond} $\Diamond_{W,F}(V)$ is a visual representation of the Deligne splitting $V_\bC = \op\,V^{p,q}_{W,F}$ (\cite[\S\ref{extn-S:gpq}]{Robles-extnHnorm}) that is given by a configuration of points in the $(p,q)$--plane that is labeled with $\tdim_\bC\,V^{p,q}_{W,F}$.  This device encodes much of the discrete data in $(W,F)$, and may illuminate the constructions and arguments  here that utilize limiting mixed Hodge structures.  In this appendix we give the Hodge diamonds for the (non-hermitian) period domain parameterizing pure, effective, weight $\sfw=2$ polarized Hodge structures on $V$ with Hodge numbers $\bfh = (2,\sfh,2)$.  There are six possible Hodge diamonds.  We have $H = \tw^2V = \fg \ot \bQ(-2)$ and $\sfn=4$.  In the diamonds below, some of the nodes are left unmarked; those missing dimensions may be determined by \cite[\eqref{extn-E:dsfilts}, \eqref{extn-E:symmetries}]{Robles-extnHnorm}.
\begin{equation} \label{E:hd2-0}
\begin{tikzpicture}[baseline={([yshift=-.5ex]current bounding box.center)}]
  \node [above] at (1.5,2.4) {$\Diamond(V)$}; 		
  \draw [<->] (0,2.75) -- (0,0) -- (2.75,0);		
  \node [left] at (2,-1) {$\sfm=4$};				
  \draw [gray] (1,0) -- (1,2);
  \draw [gray] (2,0) -- (2,2);
  \draw [gray] (0,1) -- (2,1);
  \draw [gray] (0,2) -- (2,2);
  \draw [fill] (0,2) circle [radius=0.08];
  \node [above right] at (0,2) {\scriptsize{$2$}};
  \draw [fill] (1,1) circle [radius=0.08];
  \node [above right] at (1,1) {\scriptsize{$\sfh$}};
  \draw [fill] (2,0) circle [radius=0.08];
  \node [above right] at (2,0) {\scriptsize{$2$}};
\end{tikzpicture}
\hsp{70pt}
\begin{tikzpicture}[baseline={([yshift=-.5ex]current bounding box.center)}]
  \node [above] at (1,2.5) {$\Diamond(H)$};			
  \draw [<->] (-2,2.75) -- (-2,-2) -- (2.75,-2);	
  \draw [gray] (0,-2) -- (0,2);						
  \draw [gray] (-2,0) -- (2,0);						
  \draw [gray] (2,-2) -- (2,2);
  \draw [gray] (1,-2) -- (1,2);
  \draw [gray] (-2,-2) -- (-2,2);
  \draw [gray] (-1,-2) -- (-1,2);
  \draw [gray] (-2,2) -- (2,2);
  \draw [gray] (-2,1) -- (2,1);
  \draw [gray] (-2,-1) -- (2,-1);
  \draw [gray] (-2,-2) -- (2,-2);
  \draw [fill] (-2,2) circle [radius=0.08];
  \node [above right] at (-2,2) {\scriptsize{$1$}};
  \node [below right] at (-2,2) {\footnotesize{$e_\sfd$}};
  \draw [fill] (-1,1) circle [radius=0.08];
  \node [above right] at (-1,1) {\scriptsize{$2\sfh$}};
  \draw [fill] (0,0) circle [radius=0.08];
  \node [above right] at (0,0) {};
  \draw [fill] (1,-1) circle [radius=0.08];
  \node [above right] at (1,-1) {\scriptsize{$2\sfh$}};
  \draw [fill] (2,-2) circle [radius=0.08];
  \node [above right] at (2,-2) {\scriptsize{$1$}};
  \node [below] at (2,-2) {\footnotesize{$e_0=e_\infty$}};
\end{tikzpicture}
\end{equation}
\begin{equation} \label{E:hd2-1}
\begin{tikzpicture}[baseline={([yshift=-.5ex]current bounding box.center)}]
  \node [above] at (1.5,2.4) {$\Diamond(V)$}; 		
  \draw [<->] (0,2.75) -- (0,0) -- (2.75,0);		
  \node [left] at (2,-1) {$\sfm=5$};				
  \draw [gray] (1,0) -- (1,2);
  \draw [gray] (2,0) -- (2,2);
  \draw [gray] (0,1) -- (2,1);
  \draw [gray] (0,2) -- (2,2);
  \draw [fill] (0,2) circle [radius=0.08];
  \node [above right] at (0,2) {\scriptsize{$1$}};
  \draw [fill] (0,1) circle [radius=0.08];
  \node [above right] at (0,1) {\scriptsize{$1$}};
  \draw [fill] (1,2) circle [radius=0.08];
  \node [above right] at (1,2) {\scriptsize{$1$}};
  \draw [fill] (1,1) circle [radius=0.08];
  \node [above] at (1,1) {};
  \draw [fill] (1,0) circle [radius=0.08];
  \node [above right] at (1,0) {\scriptsize{$1$}};
  \draw [fill] (2,1) circle [radius=0.08];
  \node [above right] at (2,1) {\scriptsize{$1$}};
  \draw [fill] (2,0) circle [radius=0.08];
  \node [above right] at (2,0) {\scriptsize{$1$}};
\end{tikzpicture}
\hsp{70pt}
\begin{tikzpicture}[baseline={([yshift=-.5ex]current bounding box.center)}]
  \node [above] at (1,2.5) {$\Diamond(H)$};			
  \draw [<->] (-2,2.75) -- (-2,-2) -- (2.75,-2);	
  \draw [gray] (0,-2) -- (0,2);						
  \draw [gray] (-2,0) -- (2,0);						
  \draw [gray] (2,-2) -- (2,2);
  \draw [gray] (1,-2) -- (1,2);
  \draw [gray] (-2,-2) -- (-2,2);
  \draw [gray] (-1,-2) -- (-1,2);
  \draw [gray] (-2,2) -- (2,2);
  \draw [gray] (-2,1) -- (2,1);
  \draw [gray] (-2,-1) -- (2,-1);
  \draw [gray] (-2,-2) -- (2,-2);
  \draw [fill] (-2,1) circle [radius=0.08];
  \node [above right] at (-2,1) {\scriptsize{$1$}};
  \node [below right] at (-2,1) {\footnotesize{$e_\sfd$}};
  \draw [fill] (-1,2) circle [radius=0.08];
  \node [above right] at (-1,2) {\scriptsize{$1$}};
  \draw [fill] (-1,1) circle [radius=0.08];
  \draw [fill] (-1,0) circle [radius=0.08];
  \node [above] at (-1,0) {\scriptsize{$\sfh-1$}};
  \draw [fill] (-1,-1) circle [radius=0.08];
  \node [above right] at (-1,-1) {\scriptsize{$1$}};
  \draw [fill] (0,1) circle [radius=0.08];
  \node [above right] at (0,1) {};
  \draw [fill] (0,0) circle [radius=0.08];
  \node [above right] at (0,0) {};
  \draw [fill] (0,-1) circle [radius=0.08];
  \node [above right] at (0,-1) {};
  \draw [fill] (1,-1) circle [radius=0.08];
  \node [above] at (1,-1) {};
  \draw [fill] (1,0) circle [radius=0.08];
  \node [above] at (1,0) {};
  \draw [fill] (1,1) circle [radius=0.08];
  \node [above right] at (1,1) {\scriptsize{$1$}};
  \draw [fill] (1,-2) circle [radius=0.08];
  \node [above right] at (1,-2) {\scriptsize{$1$}};
  \node [below right] at (1,-2) {\footnotesize{$e_\infty$}};
  \draw [fill] (2,-1) circle [radius=0.08];
  \node [above right] at (2,-1) {\scriptsize{$1$}};
  \node [below right] at (2,-1) {\footnotesize{$e_0$}};
\end{tikzpicture}
\end{equation}
\begin{equation} \label{E:hd2-2}
\begin{tikzpicture}[baseline={([yshift=-.5ex]current bounding box.center)}]
  \node [above] at (1.5,2.4) {$\Diamond(V)$}; 		
  \draw [<->] (0,2.75) -- (0,0) -- (2.75,0);		
  \node [left] at (2,-1) {$\sfm=6$};				
  \draw [gray] (1,0) -- (1,2);
  \draw [gray] (2,0) -- (2,2);
  \draw [gray] (0,1) -- (2,1);
  \draw [gray] (0,2) -- (2,2);
  \draw [fill] (0,2) circle [radius=0.08];
  \node [above right] at (0,2) {\scriptsize{$1$}};
  \draw [fill] (0,0) circle [radius=0.08];
  \node [above right] at (0,0) {\scriptsize{$1$}};
  \draw [fill] (1,1) circle [radius=0.08];
  \node [above right] at (1,1) {\scriptsize{$\sfh$}};
  \draw [fill] (2,2) circle [radius=0.08];
  \node [above right] at (2,2) {\scriptsize{$1$}};
  \draw [fill] (2,0) circle [radius=0.08];
  \node [above right] at (2,0) {\scriptsize{$1$}};
\end{tikzpicture}
\hsp{70pt}
\begin{tikzpicture}[baseline={([yshift=-.5ex]current bounding box.center)}]
  \node [above] at (1,2.5) {$\Diamond(H)$};			
  \draw [<->] (-2,2.75) -- (-2,-2) -- (2.75,-2);	
  \draw [gray] (0,-2) -- (0,2);						
  \draw [gray] (-2,0) -- (2,0);						
  \draw [gray] (2,-2) -- (2,2);
  \draw [gray] (1,-2) -- (1,2);
  \draw [gray] (-2,-2) -- (-2,2);
  \draw [gray] (-1,-2) -- (-1,2);
  \draw [gray] (-2,2) -- (2,2);
  \draw [gray] (-2,1) -- (2,1);
  \draw [gray] (-2,-1) -- (2,-1);
  \draw [gray] (-2,-2) -- (2,-2);
  \draw [fill] (2,0) circle [radius=0.08];
  \node [above right] at (2,0) {\scriptsize{$1$}};
  \node [below right] at (2,0) {\footnotesize{$e_0$}};
  \draw [fill] (0,2) circle [radius=0.08];
  \node [above right] at (0,2) {\scriptsize{$1$}};
  \draw [fill] (-1,1) circle [radius=0.08];
  \node [above right] at (-1,1) {\scriptsize{$\sfh$}};
  \draw [fill] (-1,-1) circle [radius=0.08];
  \node [above right] at (-1,-1) {\scriptsize{$\sfh$}};
  \draw [fill] (0,0) circle [radius=0.08];
  \node [above right] at (0,0) {};
  \draw [fill] (1,-1) circle [radius=0.08];
  \node [above right] at (1,-1) {\scriptsize{$\sfh$}};
  \draw [fill] (1,1) circle [radius=0.08];
  \node [above right] at (1,1) {\scriptsize{$\sfh$}};
  \draw [fill] (0,-2) circle [radius=0.08];
  \node [above right] at (0,-2) {\scriptsize{$1$}};
  \node [below right] at (0,-2) {\footnotesize{$e_\infty$}};
  \draw [fill] (-2,0) circle [radius=0.08];
  \node [above right] at (-2,0) {\scriptsize{$1$}};
  \node [below right] at (-2,0) {\footnotesize{$e_\sfd$}};
\end{tikzpicture}
\end{equation}
\begin{equation} \label{E:hd2-3}
\begin{tikzpicture}[baseline={([yshift=-.5ex]current bounding box.center)}]
  \node [above] at (1.5,2.4) {$\Diamond(V)$}; 		
  \draw [<->] (0,2.75) -- (0,0) -- (2.75,0);		
  \node [left] at (2,-1) {$\sfm=6$};				
  \draw [gray] (1,0) -- (1,2);
  \draw [gray] (2,0) -- (2,2);
  \draw [gray] (0,1) -- (2,1);
  \draw [gray] (0,2) -- (2,2);
  \draw [fill] (0,1) circle [radius=0.08];
  \node [above right] at (0,1) {\scriptsize{$2$}};
  \draw [fill] (1,2) circle [radius=0.08];
  \node [above right] at (1,2) {\scriptsize{$2$}};
  \draw [fill] (1,1) circle [radius=0.08];
  \node [above] at (1,1) {};
  \draw [fill] (1,0) circle [radius=0.08];
  \node [above right] at (1,0) {\scriptsize{$2$}};
  \draw [fill] (2,1) circle [radius=0.08];
  \node [above right] at (2,1) {\scriptsize{$2$}};
\end{tikzpicture}
\hsp{70pt}
\begin{tikzpicture}[baseline={([yshift=-.5ex]current bounding box.center)}]
  \node [above] at (1,2.5) {$\Diamond(H)$};			
  \draw [<->] (-2,2.75) -- (-2,-2) -- (2.75,-2);	
  \draw [gray] (0,-2) -- (0,2);						
  \draw [gray] (-2,0) -- (2,0);						
  \draw [gray] (2,-2) -- (2,2);
  \draw [gray] (1,-2) -- (1,2);
  \draw [gray] (-2,-2) -- (-2,2);
  \draw [gray] (-1,-2) -- (-1,2);
  \draw [gray] (-2,2) -- (2,2);
  \draw [gray] (-2,1) -- (2,1);
  \draw [gray] (-2,-1) -- (2,-1);
  \draw [gray] (-2,-2) -- (2,-2);
  \draw [fill] (2,0) circle [radius=0.08];
  \node [above right] at (2,0) {\scriptsize{$1$}};
  \node [below right] at (2,0) {\footnotesize{$e_0$}};
  \draw [fill] (0,2) circle [radius=0.08];
  \node [above right] at (0,2) {\scriptsize{$1$}};
  \draw [fill] (-1,1) circle [radius=0.08];
  \node [above right] at (-1,1) {\scriptsize{$4$}};
  \draw [fill] (-1,0) circle [radius=0.08];
  \node [above left] at (-1,0) {};
  \draw [fill] (-1,-1) circle [radius=0.08];
  \node [above right] at (-1,-1) {\scriptsize{$4$}};
  \draw [fill] (0,1) circle [radius=0.08];
  \node [above right] at (0,1) {};
  \draw [fill] (0,0) circle [radius=0.08];
  \node [above right] at (0,0) {};
  \draw [fill] (0,-1) circle [radius=0.08];
  \node [above right] at (0,-1) {};
  \draw [fill] (1,-1) circle [radius=0.08];
  \node [above right] at (1,-1) {\scriptsize{$4$}};
  \draw [fill] (1,0) circle [radius=0.08];
  \node [above] at (1,0) {};
  \draw [fill] (1,1) circle [radius=0.08];
  \node [above right] at (1,1) {\scriptsize{$4$}};
  \draw [fill] (0,-2) circle [radius=0.08];
  \node [above right] at (0,-2) {\scriptsize{$1$}};
  \node [below right] at (0,-2) {\footnotesize{$e_\infty$}};
  \draw [fill] (-2,0) circle [radius=0.08];
  \node [above right] at (-2,0) {\scriptsize{$1$}};
  \node [below right] at (-2,0) {\footnotesize{$e_\sfd$}};
\end{tikzpicture}
\end{equation}
\begin{equation} \label{E:hd2-4}
\begin{tikzpicture}[baseline={([yshift=-.5ex]current bounding box.center)}]
  \node [above] at (1.5,2.4) {$\Diamond(V)$}; 		
  \draw [<->] (0,2.75) -- (0,0) -- (2.75,0);		
  \node [left] at (2,-1) {$\sfm=7$};				
  \draw [gray] (1,0) -- (1,2);
  \draw [gray] (2,0) -- (2,2);
  \draw [gray] (0,1) -- (2,1);
  \draw [gray] (0,2) -- (2,2);
  \draw [fill] (0,1) circle [radius=0.08];
  \node [above right] at (0,1) {\scriptsize{$1$}};
  \draw [fill] (0,0) circle [radius=0.08];
  \node [above right] at (0,0) {\scriptsize{$1$}};
  \draw [fill] (1,2) circle [radius=0.08];
  \node [above right] at (1,2) {\scriptsize{$1$}};
  \draw [fill] (1,1) circle [radius=0.08];
  \node [above] at (1,1) {};
  \draw [fill] (1,0) circle [radius=0.08];
  \node [above right] at (1,0) {\scriptsize{$1$}};
  \draw [fill] (2,2) circle [radius=0.08];
  \node [above right] at (2,2) {\scriptsize{$1$}};
  \draw [fill] (2,1) circle [radius=0.08];
  \node [above right] at (2,1) {\scriptsize{$1$}};
\end{tikzpicture}
\hsp{70pt}
\begin{tikzpicture}[baseline={([yshift=-.5ex]current bounding box.center)}]
  \node [above] at (1,2.5) {$\Diamond(H)$};			
  \draw [<->] (-2,2.75) -- (-2,-2) -- (2.75,-2);	
  \draw [gray] (0,-2) -- (0,2);						
  \draw [gray] (-2,0) -- (2,0);						
  \draw [gray] (2,-2) -- (2,2);
  \draw [gray] (1,-2) -- (1,2);
  \draw [gray] (-2,-2) -- (-2,2);
  \draw [gray] (-1,-2) -- (-1,2);
  \draw [gray] (-2,2) -- (2,2);
  \draw [gray] (-2,1) -- (2,1);
  \draw [gray] (-2,-1) -- (2,-1);
  \draw [gray] (-2,-2) -- (2,-2);
  \draw [fill] (2,1) circle [radius=0.08];
  \node [above right] at (2,1) {\scriptsize{$1$}};
  \node [below right] at (2,1) {\footnotesize{$e_0$}};
  \draw [fill] (1,2) circle [radius=0.08];
  \node [above right] at (1,2) {\scriptsize{$1$}};
  \draw [fill] (-1,1) circle [radius=0.08];
  \node [above right] at (-1,1) {\scriptsize{$1$}};
  \draw [fill] (-1,0) circle [radius=0.08];
  \node [above] at (-1,0) {\scriptsize{$\sfh-1$}};
  \draw [fill] (-1,-1) circle [radius=0.08];
  \draw [fill] (0,1) circle [radius=0.08];
  \node [above right] at (0,1) {};
  \draw [fill] (0,0) circle [radius=0.08];
  \node [above right] at (0,0) {};
  \draw [fill] (0,-1) circle [radius=0.08];
  \node [above right] at (0,-1) {};
  \draw [fill] (1,-1) circle [radius=0.08];
  \node [above] at (1,-1) {};
  \draw [fill] (1,0) circle [radius=0.08];
  \node [above] at (1,0) {};
  \draw [fill] (1,1) circle [radius=0.08];
  \node [above] at (1,1) {};
  \draw [fill] (-1,-2) circle [radius=0.08];
  \node [above right] at (-1,-2) {\scriptsize{$1$}};
  \node [below right] at (-1,-2) {\footnotesize{$e_\infty$}};
  \draw [fill] (-2,-1) circle [radius=0.08];
  \node [above right] at (-2,-1) {\scriptsize{$1$}};
  \node [below right] at (-2,-1) {\footnotesize{$e_\sfd$}};
\end{tikzpicture}
\end{equation}
\begin{equation}\label{E:hd2-5}
\begin{tikzpicture}[baseline={([yshift=-.5ex]current bounding box.center)}]
  \node [above] at (1.5,2.4) {$\Diamond(V)$}; 		
  \draw [<->] (0,2.75) -- (0,0) -- (2.75,0);		
  \node [left] at (2,-1) {$\sfm=8$};				
  \draw [gray] (1,0) -- (1,2);
  \draw [gray] (2,0) -- (2,2);
  \draw [gray] (0,1) -- (2,1);
  \draw [gray] (0,2) -- (2,2);
  \draw [fill] (0,0) circle [radius=0.08];
  \node [above right] at (0,0) {\scriptsize{$2$}};
  \draw [fill] (1,1) circle [radius=0.08];
  \node [above right] at (1,1) {\scriptsize{$\sfh$}};
  \draw [fill] (2,2) circle [radius=0.08];
  \node [above right] at (2,2) {\scriptsize{$2$}};
\end{tikzpicture}
\hsp{70pt}
\begin{tikzpicture}[baseline={([yshift=-.5ex]current bounding box.center)}]
  \node [above] at (1,2.5) {$\Diamond(H)$};			
  \draw [<->] (-2,2.75) -- (-2,-2) -- (2.75,-2);	
  \draw [gray] (0,-2) -- (0,2);						
  \draw [gray] (-2,0) -- (2,0);						
  \draw [gray] (2,-2) -- (2,2);
  \draw [gray] (1,-2) -- (1,2);
  \draw [gray] (-2,-2) -- (-2,2);
  \draw [gray] (-1,-2) -- (-1,2);
  \draw [gray] (-2,2) -- (2,2);
  \draw [gray] (-2,1) -- (2,1);
  \draw [gray] (-2,-1) -- (2,-1);
  \draw [gray] (-2,-2) -- (2,-2);
  \draw [fill] (-2,-2) circle [radius=0.08];
  \node [above right] at (-2,-2) {\scriptsize{$1$}};
  \node [below] at (-2,-2) {\footnotesize{$e_\sfd=e_\infty$}};
  \draw [fill] (-1,-1) circle [radius=0.08];
  \node [above right] at (-1,-1) {\scriptsize{$2\sfh$}};
  \draw [fill] (0,0) circle [radius=0.08];
  \node [above right] at (0,0) {};
  \draw [fill] (1,1) circle [radius=0.08];
  \node [above right] at (1,1) {\scriptsize{$2\sfh$}};
  \draw [fill] (2,2) circle [radius=0.08];
  \node [above right] at (2,2) {\scriptsize{$1$}};
  \node [below right] at (2,2) {\footnotesize{$e_0$}};
\end{tikzpicture}
\end{equation}


\def\cprime{$'$} \def\Dbar{\leavevmode\lower.6ex\hbox to 0pt{\hskip-.23ex
  \accent"16\hss}D}

\end{document}